\documentclass[reqno,tbtags,a4paper,12pt]{amsart}

\usepackage[in]{fullpage}
\usepackage[matrix,arrow,curve]{xy}

\usepackage[matrix,arrow,curve]{xy}
\usepackage{tikz}
\usetikzlibrary{cd}
\usetikzlibrary{decorations.markings}

\tikzset{->-/.style={decoration={
     markings,
     mark=at position #1 with {\arrow{>}}},postaction={decorate}}}

\usepackage{amsmath,amssymb,amsthm,amsfonts,amscd}
\usepackage{graphics,graphicx}

\usepackage{mathtools}

\newcommand{\Z}{\mathbb{Z}}

\newcommand{\R}{\mathbb{R}}
\newcommand{\C}{\mathbb{C}}

\newcommand{\RP}{\mathbb{RP}}
\newcommand{\CP}{\mathbb{CP}}
\newcommand{\HP}{\mathbb{HP}}
\newcommand{\tHP}{\widetilde{\HP}\vphantom{\HP}}
\newcommand{\ttHP}{\widetilde{\widetilde{\HP}}\vphantom{\HP}}
\newcommand{\OP}{\mathbb{OP}}
\newcommand{\F}{\mathbb{F}}

\newcommand{\CA}{\mathcal{A}}
\newcommand{\CO}{\mathcal{O}}

\newcommand{\rC}{\mathrm{C}}
\newcommand{\rS}{\mathrm{S}}
\newcommand{\rA}{\mathrm{A}}

\newcommand{\wJ}{\widetilde{J}}

\newcommand{\fa}{\mathfrak{a}}
\newcommand{\fb}{\mathfrak{b}}
\newcommand{\fc}{\mathfrak{c}}
\newcommand{\fd}{\mathfrak{d}}
\newcommand{\fo}{\mathfrak{o}}

\newcommand{\bx}{\mathbf{x}}

\newcommand{\COa}{\mathcal{O}_{\mathrm{adm}}}

\newcommand{\CX}{\mathfrak{X}}

\newcommand{\SL}{\mathrm{SL}}
\newcommand{\SU}{\mathrm{SU}}
\newcommand{\GL}{\mathrm{GL}}
\newcommand{\GA}{\mathrm{GA}}

\newcommand{\Sym}{\mathop{\mathrm{Sym}}\nolimits}
\newcommand{\Sec}{\mathop{\mathrm{Sec}}\nolimits}

\newcommand{\Isom}{\mathop{\mathrm{Isom}}\nolimits}
\newcommand{\link}{\mathop{\mathrm{link}}\nolimits}
\newcommand{\cost}{\mathop{\mathrm{cost}}\nolimits}

\newtheorem{theorem}{Theorem} [section]
\newtheorem{propos}[theorem] {Proposition}
\newtheorem{cor}[theorem] {Corollary}
\newtheorem{lem}[theorem]{Lemma}
\newtheorem{conj}[theorem] {Conjecture}
\newtheorem{problem}[theorem] {Problem}
\newtheorem{question}[theorem] {Question}
\theoremstyle{definition}
\newtheorem{remark}[theorem]{Remark}
\newtheorem{defin}[theorem]{Definition}

\numberwithin{equation}{section}
\numberwithin{table}{section}

\author{Alexander A. Gaifullin}

\address{Steklov Mathematical Institute of Russian Academy of Sciences, Moscow, Russia}
\address{Skolkovo Institute of Science and Technology, Moscow, Russia}
\address{Lomonosov Moscow State University, Moscow, Russia}
\address{Institute for the Information Transmission Problems of the Russian Academy of Sciences (Kharkevich Institute), Moscow, Russia}
\email{agaif@mi-ras.ru}

\thanks{This work was performed at the Steklov International Mathematical Center and supported by the Ministry of Science and Higher Education of the Russian Federation (agreement no. 075-15-2022-265).}

\title{634 vertex-transitive and more than $\boldsymbol{10^{103}}$ non-vertex-transitive 27-vertex triangulations of manifolds like the octonionic projective plane}
\date{}

\keywords{Minimal triangulation, octonionic projective plane, manifold like a projective plane, K\"uhnel triangulation, Brehm--K\"uhnel triangulations, vertex-transitive triangulation, combinatorial manifold}

\subjclass[2020]{57Q15, 57Q70, 05E45}

\begin{document}
\begin{abstract}
In 1987 Brehm and K\"uhnel showed that any combinatorial $d$-manifold with less than $3d/2+3$ vertices is PL homeomorphic to the sphere and any combinatorial $d$-manifold with exactly $3d/2+3$ vertices is PL homeomorphic to either the sphere or a manifold like a projective plane in the sense of Eells and Kuiper. The latter possibility may occur for $d\in\{2,4,8,16\}$ only. There exist a unique $6$-vertex triangulation of~$\RP^2$, a unique $9$-vertex triangulation of~$\CP^2$, and at least three $15$-vertex triangulations of~$\HP^2$. However, until now, the question of whether there exists a $27$-vertex triangulation of a manifold like the octonionic projective plane has remained open. We solve this problem by constructing a lot of examples of such triangulations. Namely, we construct $634$ vertex-transitive $27$-vertex combinatorial $16$-manifolds like the octonionic projective plane. Four of them have symmetry group $\rC_3^3\rtimes \rC_{13}$ of order~$351$, and the other $630$ have symmetry group~$\rC_3^3$ of order~$27$. Further, we construct more than $10^{103}$ non-vertex-transitive $27$-vertex combinatorial $16$-manifolds like the octonionic projective plane. Most of them have trivial symmetry group, but there are also symmetry groups~$\rC_3$, $\rC_3^2$, and~$\rC_{13}$. We conjecture that all the triangulations constructed are PL homeomorphic to the octonionic projective plane~$\OP^2$. Nevertheless, we have no proof of this fact so far.
\end{abstract}

\maketitle

\section{Introduction}\label{section_intro}

Recall that a \textit{combinatorial $d$-manifold} is a simplicial complex in which the link of every vertex is PL homeomorphic to the $(d-1)$-dimensional sphere.

In 1987 Brehm and K\"uhnel~\cite{BrKu87} proved that any combinatorial $d$-manifold with less than $3d/2+3$ vertices is PL homeomorphic to the sphere. Moreover, they proved that combinatorial manifolds~$K$ that have exactly $3d/2+3$ vertices and are not homeomorphic to~$S^d$ may exist only in dimensions $d\in\{2,4,8,16\}$. Such combinatorial manifolds must satisfy the following conditions:
\begin{itemize}
\item[(a)] If $d=2$, then $K$ is isomorphic to the $6$-vertex triangulation~$\RP^2_6$ of the real projective plane obtained by taking the quotient of the boundary of the regular icosahedron by the antipodal involution. If $d=4$, then $K$ is isomorphic to the K\"uhnel $9$-vertex triangulation~$\CP^2_9$ of the complex projective plane, see~\cite{KuBa83}, \cite{KuLa83}. If $d=8$ or $d = 16$, then $K$ is a \textit{manifold like a projective plane} in the sense of Eells and Kuiper~\cite{EeKu62}, which means that $K$ admits a PL Morse function with exactly three critical points. In particular, this implies that $K$ admits a CW decomposition into three cells of dimensions~$0$, $d/2$, and~$d$ and therefore
$$
H^*(K;R)\cong R[x]/\bigl(x^3\bigr),\qquad \deg x = \frac{d}{2}\,,
$$
for any ring~$R$.

\item[(b)] $K$ is \textit{$s$-neighborly} with $s=d/2+1$, which means that every $s$-element set of vertices spans a simplex.
\item[(c)] $K$ satisfies the following \textit{complementarity} (or \textit{duality}) \textit{condition}: For each subset~$W$ of the set of vertices~$V$ of~$K$, exactly one of the sets~$W$ or~$V\setminus W$ spans a simplex of~$K$.
\item[(d)] The numbers~$f_k$ of $k$-simplices of~$K$ are as shown in Table~\ref{table_f}.
\begin{table}[t]
\caption{The face numbers of combinatorial $d$-manifolds that have $3d/2+3$ vertices and are not homeomorphic to the sphere}\label{table_f}
\begin{align*}
\boldsymbol{d}&\boldsymbol{=2\!:}& \boldsymbol{d}&\boldsymbol{=4\!:} 
&\boldsymbol{d}&\boldsymbol{=8\!:} & \boldsymbol{d}&\boldsymbol{=16\!:}\\
f_0&= 6 & f_0 &=9 & f_0&=15 & f_0&=27& f_9 &=8335899\\
f_1&=15 & f_1&=36 & f_1&=105 & f_1&=351&   f_{10} &=12184614\\
f_2&=10 & f_2&=84 & f_2&=455 & f_2&=2925&   f_{11} &=14074164\\
&&           f_3&=90 & f_3 &=1365 & f_3&=17550&  f_{12} &=12301200\\
&& f_4&= 36 & f_4&= 3003 & f_4&= 80730& f_{13} &=7757100\\
&&&& f_5 &= 4515 & f_5&=296010& f_{14} &=3309696\\
&&&& f_6 &= 4230 & f_6&= 888030& f_{15} &=853281\\
&&&& f_7&= 2205 & f_7&=2220075& f_{16} &=100386\\
&&&& f_8&=490 & f_8&=4686825 &&
\end{align*}
\end{table}

\item[(e)] The canonical embedding $K\subset \R^{3d/2+2}$ obtained by regarding $K$ as the subcomplex of the standard $(3d/2+2)$-simplex is \textit{tight}. This means that, for each open half-space $h\subset \R^{3d/2+2}$ and any field of coefficients~$F$, the inclusion $K\cap h \subset K$ induces an injective map in homology $H_*(K\cap h;F)\to H_*(K;F)$; for more information on tight embeddings, see~\cite{Kuh95}.
\end{itemize}

The uniqueness of the $6$-vertex combinatorial $2$-manifold that is not a sphere is checked easily. The uniqueness of the $9$-vertex combinatorial $4$-manifold that is not a sphere was first proved in~\cite{KuLa83}; for simpler proofs, see~\cite[\S 19]{ArMa91},~\cite{BaDa01}. Condition~(a) for $d=8$ or $d=16$ is proved in~\cite{BrKu87}, the complementarity condition~(c) is proved in~\cite[\S 17,\,\S 20]{ArMa91},  the neighborliness condition~(b) follows immediately from~(c), since there are no simplices of dimensions greater than~$d$, and the tightness condition~(e) is deduced from~(c) in~\cite[Proposition~3]{BrKu92}. The face numbers~$f_i(K)$ can be computed using the complementarity condition, the Dehn--Sommerville equations, and the known value of the Euler characteristic $\chi(K)=3$ (for $d>2$). This computation can be found in~\cite{BrKu92} for $d=8$ and in~\cite[Section~4C]{Kuh95} for $d=16$, cf.~\cite{ChMa13}.

In the case~$d=8$, Brehm and K\"uhnel~\cite{BrKu92} constructed three $15$-vertex combinatorial manifolds like the quaternionic projective plane that satisfy conditions~(a)--(e). Nevertheless, they could not decide whether these combinatorial manifolds are homeomorphic to the quaternionic projective plane~$\HP^2$. This problem remained open for a long time until Gorodkov proved that the Brehm--K\"uhnel complexes are indeed PL homeomorphic to~$\HP^2$, see~\cite{Gor16},~\cite{Gor19}. 

For a simplicial complex~$K$, we denote by $\Sym(K)$ the \textit{symmetry group} of~$K$, that is, the group consisting of all permutations of vertices of~$K$ that take simplices to simplices and non-simplices to non-simplices. A simplicial complex is called \textit{vertex-transitive} if its symmetry group acts transitively on the vertices. Vertex-transitive triangulations of manifolds are especially interesting, see~\cite{KoLu05}.
Note that~$\RP^2_6$ and $\CP^2_9$ are vertex-transitive. In the case $d=8$, one of the three Brehm--K\"uhnel triangulations is vertex-transitive, and the other two are not. We denote the vertex-transitive triangulation by~$\HP^2_{15}$. Moreover, Brehm proved that $\HP^2_{15}$ is the only vertex-transitive $15$-vertex triangulation of a $8$-manifold like the quaternionic projective plane, cf.~\cite[Corollary~11]{KoLu05}. 

Using his program \texttt{BISTELLAR}, Lutz~\cite{Lut05} constructed three more $15$-vertex triangulations of~$\HP^2$. Unfortunately, their lists of simplices are not available via the link in~\cite{Lut05} and are apparently lost. Therefore, it is hardly possible to say anything about the properties of these three triangulations. After writing the present paper, the author~\cite{Gai23b} constructed a lot of new $15$-vertex triangulations of~$\HP^2$ with different symmetry groups.

Until now, the case $d=16$ has remained completely open. Namely, there were no candidates for a $27$-vertex combinatorial triangulation of the octonionic projective plane~$\OP^2$ or a manifold like the octonionic projective plane. (Note that the octonionic projective plane is also often called the \textit{Cayley plane}.) The problem of whether there exists a $27$-vertex combinatorial $16$-manifold like the octonionic projective plane and (if exists) whether it is homeomorphic to~$\OP^2$ was posed in~\cite{BrKu87} (cf.~\cite{ArMa91}, \cite{BrKu92},  \cite{Kuh95}, \cite{Lut05}, \cite{ChMa13}). In the present paper we solve the first part of this problem. Namely, we construct  $27$-vertex combinatorial $16$-manifolds like the octonionic projective plane. However, we are in the same situation as Brehm and K\"uhnel were in 1992 in the quaternionic case---we cannot decide yet whether our simplicial complexes are indeed homeomorphic to~$\OP^2$.

Kramer~\cite{Kra03} classified PL manifolds like a projective plane in dimensions~$8$ and~$16$, see Remark~\ref{remark_classification}  for more detail. He showed that in dimension~$8$ such manifolds are distinguished by their Pontryagin numbers, and in dimension~$16$ by their Pontryagin numbers and certain exotic PL characteristic numbers. Gorodkov's result in the quaternionic case (see~\cite{Gor16},~\cite{Gor19}) was based on the computation of the first  Pontryagin class of the Brehm--K\"uhnel combinatorial manifolds using an explicit combinatorial formula due to the author of the present paper, cf.~\cite{Gai04}, \cite{Gai08}, \cite{Gai10}, \cite{GaGo19}. However, to decide whether a $16$-manifold like a projective plane is indeed homeomorphic to~$\OP^2$ one needs to compute the second Pontryagin class and the first exotic PL characteristic class of it. There is currently no known efficient way to do this.

We denote by~$\rC_n$ the cyclic group of order~$n$, by~$\rS_n$ the symmetric group of degree~$n$, and by~$\rA_n$ the alternating group of degree~$n$.

\begin{theorem}\label{theorem_main}
There are at least
$$
\frac1{351}\cdot\left(2^{351}+13\cdot 2^{118}+81\cdot 2^{29}\right)+2 \approx 1.3\cdot 10^{103}
$$
combinatorially distinct $27$-vertex combinatorial triangulations of  $16$-manifolds like the octonionic projective plane. The symmetry groups of these triangulations are $\rC_3^3\rtimes \rC_{13}$, $\rC_3^3$, $\rC_{13}$, $\rC_3^2$, $\rC_3$, and the trivial group~$1$; the numbers of triangulations with each symmetry group are given in Table~\ref{table_num_triang}. The $634$ triangulations with symmetry groups~$\rC_3^3\rtimes \rC_{13}$ and $\rC_3^3$ are vertex-transitive, and all other triangulations are not. In addition, the four triangulations~$K_1$, $K_2$, $K_3$, and~$K_4$  with the largest symmetry group~$\rC_3^3\rtimes \rC_{13}$ have the following properties:
\begin{enumerate}
\item $\Sym(K_i)$ acts transitively and freely on the $351$ undirected edges of~$K_i$,
\item  $\Sym(K_i)$ acts freely on the set of $16$-simplices of~$K_i$, so $K_i$ contains exactly $286$ orbits of $16$-dimensional simplices, each consisting of $351$ simplices.
\end{enumerate}
Moreover, $K_1$, $K_2$, $K_3$, and~$K_4$ are the only (up to isomorphism) $27$-vertex combinatorial triangulations of $16$-manifolds like a projective plane whose symmetry groups contain a subgroup isomorphic to $\rC_3^3\rtimes \rC_{13}$.

Two of the four combinatorial manifolds $K_1$, $K_2$, $K_3$, and~$K_4$ (namely, $K_2$ and $K_3$ in our notation), and all mentioned above combinatorial manifolds with smaller symmetry groups are PL homeomorphic to each other.
\end{theorem}

\begin{table}
\caption{Numbers of $27$-vertex triangulations of $16$-manifolds like the octonionic projective plane with the given symmetry groups}\label{table_num_triang}
\begin{tabular}{|c|l|c|}
\hline
Symmetry group & Number of triangulations \\
\hline
$\vphantom{\strut_3^3}\rC_3^3\rtimes \rC_{13}$ & $4$\\
\hline
$\vphantom{\strut_3^3}\rC_3^3$ & $\ge 630$  \\
\hline  
$\vphantom{\strut_3^3}\rC_{13}$ &  $\ge 2^{27}-2$\\
\hline  
$\vphantom{\strut_3^3}\rC_{3}^2$ & $\ge (2^{39}-2^{13})/3$\\
\hline  
$\vphantom{\strut_3^3}\rC_{3}$ & $\ge (2^{117}-2^{41}+3\cdot 2^{13})/9$\\
\hline  
$\vphantom{\strut_3^3}1$ & $\ge \bigl(2^{351}-13\cdot 2^{117}+39\cdot 2^{39}- 27\cdot (2^{27}+2^{13}-2)\bigr)/351$ \\
\hline
\end{tabular}
\end{table}

\begin{remark}
For the two other combinatorial manifolds~$K_1$ and~$K_4$, we do not know whether they are PL homeomorphic to~$K_2$ and~$K_3$ or to each other.
\end{remark}

\begin{remark}
 We will construct exactly as many triangulations with each symmetry group as indicated in the second column of Table~\ref{table_num_triang}. Signs~$\ge$ mean that there may be other triangulations with the same symmetry groups. On the contrary, the absence of sign~$\ge$ in the first line means that we will prove that there are no other triangulations with the symmetry group~$\rC_3^3\rtimes \rC_{13}$, cf. Remark~\ref{remark_no_other}.
\end{remark}

As we have mentioned in the beginning of this section, conditions~(a)--(e)  hold whenever a combinatorial $d$-manifold with $3d/2+3$ vertices is not homeomorphic to the sphere. Hence, all the triangulations in Theorem~\ref{theorem_main} satisfy these conditions.

Note that there is a unique (up to isomorphism) semi-direct product $\rC_3^3\rtimes \rC_{13}$ that is not direct, so we may not specify the action of~$\rC_{13}$ on~$\rC_3^3$. An explicit description of the four simplicial complexes~$K_1$, $K_2$, $K_3$, and~$K_4$ and their symmetry group will be given in Section~\ref{section_result}, see Theorem~\ref{theorem_main_detail}. All the triangulations in Theorem~\ref{theorem_main} with smaller symmetry groups are obtained from~$K_2$ and~$K_3$ by \textit{triple flips} i.\,e. the replacements of a subcomplex of the form
\begin{equation}\label{eq_sc}
(\Delta_1*\partial\Delta_2)\cup(\Delta_2*\partial\Delta_3)\cup(\Delta_3*\partial\Delta_1),
\end{equation}
where $\Delta_1$, $\Delta_2$, and~$\Delta_3$ are $8$-simplices, with the subcomplex
$$
(\partial\Delta_1*\Delta_2)\cup(\partial\Delta_2*\Delta_3)\cup(\partial\Delta_3*\Delta_1),
$$
see Section~\ref{section_moves}. The same method does not work for~$K_1$ and~$K_4$, since they contain no subcomplexes of the form~\eqref{eq_sc}. Note that the same triple flips (with $\dim\Delta_i=4$) were  used by Brehm and K\"uhnel to construct  two $15$-vertex triangulations of manifolds like the quaternionic projective plane with smaller symmetry groups starting from the most symmetric triangulation with symmetry group~$\rA_5$.

\begin{remark}\label{remark_classification} 
Recall the classification of PL manifolds like projective planes up to PL homeomorphism. In dimension~$2$ any such manifold is obviously PL homeomorphic to~$\RP^2$. In dimension~$4$, it follows from Freedman's classification of topological simply-connected $4$-manifolds that any PL manifold like the complex projective plane is homeomorphic to~$\CP^2$. However, the problem of classification of PL (equivalently, smooth) structures on~$\CP^2$ is still open. In dimensions $8$ and~$16$, there exist countably many pairwise non-homeomorphic PL manifolds like projective planes. First examples of such manifolds are due to Milnor~\cite{Mil56} in dimension~$8$ and Shimada~\cite{Shi57} in dimension~$16$. Namely, if $d=2m$, where $m$ is either $4$ or~$8$, then one can consider $m$-dimensional real vector bundles~$\xi$ over the sphere~$S^m$. If the Euler number $\langle e(\xi),[S^m]\rangle$ is $\pm 1$ (one can make it to be~$1$ by reversing the orientation of~$\xi$), then the Thom space~$Y=T(\xi)$ has a natural structure of a $d$-dimensional PL manifold and is a manifold like a projective plane. Oriented $\R^m$-vector bundles over $S^m$ with $e(\xi)=1$ are classified by their Pontryagin numbers $\langle p_1(\xi),[S^4]\rangle$ and $\langle p_2(\xi),[S^8]\rangle$, respectively, which can take any integral value ${}\equiv 2\pmod 4$ in the case $m=4$ and any integral value ${}\equiv 6\pmod{12}$ in the case $m=8$. This construction yields a countable series of PL $d$-manifolds~$Y^d_h=Y^d_{1-h}$ like a projective plane indexed by $h\in\Z$ with characteristic numbers
\begin{align}
p_1^2\left[Y^8_h\right] &= \left\langle p_1^2\left(Y^8_h\right),\left[Y^8_h\right] \right\rangle = 4(2h-1)^2,\label{eq_pont8}\\
p_2^2\left[Y^{16}_h\right]&=\left\langle p_2^2\left(Y^{16}_h\right),\left[Y^{16}_h\right] \right\rangle = 36(2h-1)^2.\label{eq_pont16}
\end{align}
(These values of the characteristic numbers correspond to the orientation of~$Y^d_h$ chosen so that the signature of~$Y^d_h$ is $1$.) The Hirzebruch signature theorem implies that  the $L$-genus of any manifold like a projective plane is equal to~$1$, so all Pontryagin numbers of~$Y^d_h$ are uniquely determined by~\eqref{eq_pont8},~\eqref{eq_pont16}. The standard projective planes~$\HP^2$ and~$\OP^2$ correspond to $h=1$. Eells and Kuiper~\cite{EeKu62} explored the manifolds~$Y^d_h$ in more detail. In particular, they proved that there are exactly $6$ and $60$ homotopy classes of such manifolds in the cases $d=8$ and $d=16$, respectively. A complete classification of PL manifolds like projective planes in dimensions~$8$ and~$16$ was obtained by Kramer~\cite{Kra03}:
\begin{itemize}
\item In dimension~$8$, he showed that any such manifold is PL homeomorphic to one of Milnor's manifolds~$Y^8_h$.
\item On the contrary, in dimension~$16$, he showed that Shimada's manifolds~$Y^{16}_h$ do not exhaust all PL manifolds like the octonionic projective plane. Namely, in the construction described above, one can replace the vector bundle~$\xi$ with an arbitrary PL $\R^8$-bundle~$\xi$ over~$S^8$. For such bundles, the second Pontryagin class~$p_2(\xi)$ should no longer be integral. However, the class $\frac76p_2(\xi)$ is always integral. Besides, in dimension~$8$, there is the first exotic PL characteristic class $\kappa\in H^8(\mathrm{BPL};\Z/4\Z)$, which is not a reduction of any polynomial in Pontryagin classes. Kramer constructed a  series of PL manifolds~$Y_{h,\theta}^{16}=Y_{1-h,\theta}^{16}$ like the octonionic projective plane indexed by $h\in\frac17\Z$ and $\theta\in\Z/4\Z$ with characteristic numbers
\begin{gather*}
p_2^2\left[Y^{16}_{h,\theta}\right] = 36(2h-1)^2,\qquad
\left(\left(\frac76p_2\right)\kappa\right) \left[Y^{16}_{h,\theta}\right]  = \theta,
\end{gather*}
and proved that any PL manifold like the octonionic projective plane is PL homeomorphic to one of these manifolds. (The manifolds $Y_{h,0}^{16}$ with $h\in\Z$ are exactly Shimada's manifolds.)
\end{itemize}
Thus, our combinatorial manifolds from  Theorem~\ref{theorem_main} are PL homeomorphic to some of Kramer's manifolds~$Y_{h,\theta}^{16}$.
\end{remark}

\begin{conj}\label{conj_main}
The combinatorial manifolds~$K_1$, $K_2$, $K_3$, and~$K_4$ are PL homeomorphic to~$\OP^2$.
\end{conj}

\begin{remark}
There is a nice and intriguing parallel between $(3d/2+3)$-vertex triangulations of manifolds not homeomorphic to the sphere and theory of Severi varieties in complex projective geometry. This parallel was noticed and explored by Chapoton and Manivel~\cite{ChMa13}. Nevertheless, there are many still open questions and numerical coincidences that are yet to be explained. Let $X\subset \CP^{n-1}$ be a smooth irreducible complex projective variety of dimension~$d$ that is not contained in a hyperplane. The \textit{secant variety} $\Sec(X)$ is the Zariski closure of the union of the lines joining any two points of~$X$. The following result is due to Zak~\cite{Zak81},~\cite{Zak85}:
\begin{enumerate}
\item If $n < 3d/2+3$, then $\Sec(X)=\CP^{n-1}$.
\item If $n= 3d/2+3$, then either $\Sec(X)=\CP^{n-1}$ or $d\in\{2,4,8,16\}$. Moreover, there are exactly four examples, one for every $d=2,4,8,16$, of varieties with $\Sec(X)\ne \CP^{n-1}$:
\begin{itemize}
\item the Veronese surface $\CP^2\hookrightarrow\CP^5$,
\item the Segre embedding $\CP^2\times\CP^2\hookrightarrow\CP^8$,
\item the Pl\"ucker embedding $G_{\C}(2,6)\hookrightarrow\CP^{14}$
\item the $16$-dimensional variety $E\subset \CP^{26}$ discovered by Lazarsfeld~\cite{Laz81}.
\end{itemize} 
\end{enumerate}

Note that the four varieties~$\CP^2$, $\CP^2\times\CP^2$, $G_{\C}(2,6)$, and~$E$ can be naturally regarded as the complexifications of the projective planes~$\RP^2$, $\CP^2$, $\HP^2$, and~$\OP^2$, respectively (see, e.\,g.,~\cite{LaMa01}).
\end{remark}

\begin{remark}\label{remark_Ale}
It is not so surprising that the group $\rC_3^3\rtimes \rC_{13}$ arises as the symmetry group of combinatorial manifolds that are most likely homeomorphic to~$\OP^2$. The matter is that this group can be naturally realized as a subgroup in the isometry group~$\Isom(\OP^2)$ of~$\OP^2$ endowed with the standard Fubini--Study metric. By a result of Borel~\cite{Bor50} (cf.~\cite{Bae01} and references therein), the group~$\Isom(\OP^2)$ is isomorphic to the $52$-dimensional exceptional simply connected compact Lie group~$F_4$. Alekseevskii~\cite{Ale74} studied Jordan subgroups of complex Lie groups. We are not going to give a precise definition of a Jordan subgroup here. Let us only mention that any Jordan subgroup~$A$ of a Lie group~$G$ is a commutative finite subgroup with finite normalizer~$N_G(A)$. In particular, this means that $A$ is not contained in any torus in~$G$. In the case of~$F_4$, Alekseevskii proved that the only (up to conjugation) Jordan subgroup~$A$ is isomorphic to~$\rC_3^3$ and the normalizer~$N_{F_4}(A)$ is isomorphic to~$\SL(3,\F_3)$ and acts in the standard way on~$A\cong \rC_3^3\cong\F_3^3$. This group~$A$ was later rediscovered and studied by several authors, see~\cite{Gri91} and references therein. Hence the group $\rC_3^3\rtimes \SL(3,\F_3)$ becomes realized as a finite isometry group of~$\OP^2$. Choosing a  subgroup of~$\SL(3,\F_3)$ isomorphic to~$\rC_{13}$, we arrive to an isometry group of~$\OP^2$ isomorphic to~$\rC_3^3\rtimes \rC_{13}$.
\end{remark}

Recall that, for~$\RP^2_6$, $\CP^2_9$, and the three Brehm--K\"uhnel triangulations~$\HP^2_{15}$, $\tHP^2_{15}$  and~$\ttHP^2_{15}$, we have
\begin{gather*}
\Sym (\RP^2_6)\cong \rA_5, \qquad
\Sym (\CP^2_9)\cong \rC_3^2\rtimes \rC_6\cong\mathrm{He}_3\rtimes \rC_2, \\
\Sym (\HP^2_{15})\cong \rA_5, \qquad
\Sym \left(\tHP^2_{15}\right)\cong \rA_4, \qquad
\Sym \left(\ttHP^2_{15}\right)\cong \rS_3,
\end{gather*}
where $ \mathrm{He}_3$ is the Heisenberg group of $3\times 3$ upper unitriangular matrices with entries in~$\F_3$, see~\cite{KuBa83},~\cite{BrKu92}. In the semi-direct product $\rC_3^2\rtimes \rC_6$ the action of the generator of~$\rC_6$ on the standard generators~$a$ and~$b$ of~$\rC_3^2$ is given by $a\mapsto a^{-1}b^{-1}$, $b\mapsto b^{-1}$. In the semi-direct product $\mathrm{He}_3\rtimes \rC_2$ the action of the generator of~$\rC_2$ on~$\mathrm{He}_3$ is given by
$$
\begin{pmatrix}
1 & x & z\\
0 & 1 & y\\
0 & 0 & 1
\end{pmatrix}
\mapsto
\begin{pmatrix*}[r]
1 & -x & -z\\
0 & 1 & y\\
0 & 0 & 1
\end{pmatrix*}\,.
$$
We see that there is a strange alternation. The symmetry groups of~$\RP^2_6$ and~$\HP^2_{15}$ are simple, while the symmetry groups of~$\CP^2_9$ and each of our triangulations~$K_i$ are solvable. We do not know any explanation for this phenomenon.

\begin{problem}\label{probl_enum}
Enumerate all $15$-vertex $8$-dimensional and $27$-vertex $16$-dimensional combinatorial manifolds like a projective plane. Which groups can occur as the symmetry groups of such combinatorial manifolds? 
\end{problem}

In the case $d=16$, this problem seems to be hard, since by Theorem~\ref{theorem_main} the number of combinatorially distinct $27$-vertex combinatorial $16$-manifolds is huge. The problem that seems to be accessible is as follows.

\begin{problem}\label{probl_enum_trans}
Enumerate all vertex-transitive $27$-vertex combinatorial $16$-manifolds like the octonionic projective plane. Which groups can occur as the symmetry groups of such combinatorial manifolds? 
\end{problem}

\begin{remark}
After writing the present paper, the author obtained a complete answer to the question about possible symmetry groups of $15$-vertex $8$-dimensional combinatorial manifolds like a projective plane and a partial classification of such combinatorial manifolds, see~\cite{Gai23b}. Also, the author obtained a partial result towards enumeration of possible symmetry groups of $27$-vertex $16$-dimensional combinatorial manifolds like a projective plane, see~\cite{Gai23a}.
\end{remark}

Note that the transitivity of the action of~$\Sym(K_i)$ on the set of undirected edges of~$K_i$ is a completely new phenomenon. Neither $\CP^2_9$ nor any of $15$-vertex triangulations of~$\HP^2$ possesses this property. (The case of~$\RP^2_6$ is very special; the group $\Sym(\RP^2_6)$ acts transitively even on the set of directed edges and, moreover, on the set of all complete flags of simplices $\sigma^0\subset\sigma^1\subset\sigma^2$.)

Also, a new phenomenon is that $\Sym(K_i)$ acts freely on the set of maximal simplices. None of the previously known triangulations of manifolds like a projective plane possesses this property. Namely,  
\begin{itemize}
\item the $60$-element group $\Sym(\RP^2_6)$ acts transitively on the $10$-element set of $2$-simplices, 
\item the $54$-element group $\Sym(\CP^2_9)$ acts on the $36$-element set of $4$-simplices with two orbits consisting of $27$ and~$9$ simplices, respectively,
\item the $60$-element group $\Sym(\HP^2_{15})$ acts on the $490$-element set of $8$-simplices with $5$ orbits of $60$ elements, $4$ orbits of $30$ elements, $1$ orbit of $20$ elements, $2$ orbits of $15$ elements, and $2$ orbits of~$10$ elements,
\item the $12$-element group $\Sym\left(\tHP^2_{15}\right)$ acts on the $490$-element set of $8$-simplices with $36$ orbits of $12$~elements, $6$~orbits of $6$~elements, $4$~orbits of $4$~elements, and $2$~orbits of $3$~elements.
\item the $6$-element group $\Sym\left(\ttHP^2_{15}\right)$ acts on the $490$-element set of $8$-simplices with $75$ orbits of $6$~elements, $12$~orbits of $3$~elements, $1$~orbit of $2$~elements, and $2$~orbits of $1$~element.
\end{itemize}

\begin{remark}
A \textit{minimal triangulation} of a manifold is a triangulation of this manifold with the smallest possible number of vertices. There are not so many classes of manifolds for which minimal triangulations are known, see~\cite{Lut05}. Certainly, $\RP^2_6$, $\CP^2_9$, all $15$-vertex triangulations of~$\HP^2$, and all our triangulations from Theorem~\ref{theorem_main} are minimal, since all triangulations with smaller numbers of vertices are spheres. Alongside with the problem of finding minimal triangulations, there is an interesting problem of finding triangulations that are minimal in the class of triangulations admitting a regular coloring of vertices in $d+1$ colors, where $d$ is the dimension. In~\cite{Gai09}, the author constructed triangulations~$\RP^2_9$ and~$\CP^2_{15}$ that are minimal in this class; they have symmetry groups $\rS_3\times \rC_2$ and $\rS_4\times \rS_3$, respectively. It would be interesting to study the corresponding problem for~$\HP^2$ and~$\OP^2$ or manifolds like projective planes. In particular, it would be nice to understand if the equality of the numbers of vertices of the triangulations~$\RP^2_9$ and~$\CP^2_{15}$ with the numbers of vertices of the minimal triangulations~$\CP^2_9$ and~$\HP^2_{15}$, respectively, is a coincidence or has a reason.
\end{remark}

The present paper is organized as follows. In Section~\ref{section_result} we give a more detailed formulation of the part of Theorem~\ref{theorem_main} concerning the four triangulations~$K_1$, $K_2$, $K_3$, and~$K_4$ with the largest symmetry group $\rC_3^3\rtimes \rC_{13}$ (Theorem~\ref{theorem_main_detail}). In Section~\ref{section_scheme} we give the scheme of the proof of Theorem~\ref{theorem_main_detail}. Sections~\ref{section_find}--\ref{section_no_add_sym} contain a detailed proof of Theorem~\ref{theorem_main_detail} following this scheme. In Section~\ref{section_moves} we construct triangulations with smaller symmetry groups by applying triple flips to~$K_2$ and~$K_3$, and prove that all the obtained triangulations are PL homeomorphic to each other and to~$K_2$ and~$K_3$ (Theorem~\ref{theorem_smaller_groups}). The proof of Theorem~\ref{theorem_smaller_groups} is given in Section~\ref{section_proof_smaller_groups}. Theorems~\ref{theorem_main_detail} and~\ref{theorem_smaller_groups} together constitute Theorem~\ref{theorem_main}.

In Sections~\ref{section_fixed} and~\ref{section_links} we explore some properties of the triangulations~$K_1$, $K_2$, $K_3$, and~$K_4$. Namely, in Section~\ref{section_fixed} we study the fixed point sets~$K_i^H$ of subgroups $$H\subset\Sym(K_i)=\rC_3^3\rtimes \rC_{13}.$$
The most interesting result is that $K_i^{\rC_3}$ is the K\"uhnel triangulation~$\CP^2_9$. In Section~\ref{section_links} we study the $15$-dimensional combinatorial spheres~$L_1$, $L_2$, $L_3$, and~$L_4$ that arise as the links of vertices of~$K_1$, $K_2$, $K_3$, and~$K_4$, respectively. The main result is that $L_1$ and~$L_4$ are \textit{unflippable} i.\,e. admit no bistellar move, except for insertions of vertices.

Most of our proofs are computer assisted. In the paper we describe the required algorithms and present the results of their implementations. Explicit implementations of all the algorithms used in the present paper can be found in~\cite{Gai-prog}. 

This work arose from multiple  discussions with Denis Gorodkov, Alexander Kuznetsov, and Dmitri Orlov of the results of~\cite{ChMa13}. I am highly indebted to them for these really fruitful discussions. I am also grateful to Semyon Abramyan, Victor Buchstaber, Vasilii Rozhdestvenskii, Constantin Shramov, Igor Spiridonov, and Alexander Veselov for useful comments.

\section{Triangulations with symmetry group $\rC_3^3\rtimes \rC_{13}$}\label{section_result}

Recall terminology concerning abstract simplicial complexes. An (\textit{abstract}) \textit{simplicial complex} on the vertex set~$V$ is a set~$K$ of subsets of~$V$ such that 
\begin{itemize}
\item $\varnothing\in K$,
\item if $\sigma\in K$ and~$\tau\subset\sigma$, then $\tau\in K$.
\end{itemize}
We will assume that all simplicial complexes under consideration contain no \textit{ghost vertices}, that is, all $1$-element subsets of~$V$ belong to~$K$. 
The \textit{dimension} of a simplex $\sigma\in K$ is, by definition, the cardinality of~$\sigma$ minus~$1$. The \textit{link} of a simplex $\sigma\in K$ is the simplicial complex
$$
\link(\sigma,K)=\bigl\{\tau\subseteq V\setminus\sigma \colon \sigma\cup\tau\in K \bigr\}.
$$

If $K_1$ and~$K_2$ are simplicial complexes on the vertex sets~$V_1$ and~$V_2$, then an \textit{isomorphism} of~$K_1$ and~$K_2$ is a bijection $f\colon V_1\to V_2$ that takes simplices of~$K_1$ to simplices of~$K_2$ and non-simplices of~$K_1$ to non-simplices of~$K_2$. The automorphisms of a simplicial complex~$K$ will be called \textit{symmetries} of~$K$; we denote the symmetry group of~$K$ by~$\Sym(K)$.

A simplicial complex~$K$ is said to be \textit{pure of dimension~$d$} if any maximal (with respect to the inclusion) simplex of~$K$ is $d$-dimensional. A simplicial complex~$K$ is called a \textit{weak $d$-pseudomanifold} if $K$ is pure of dimension~$d$ and any $(d-1)$-simplex of~$K$ is contained in exactly two $d$-simplices of~$K$. A simplicial complex~$K$ is called a \textit{combinatorial $d$-sphere} (respectively, a \textit{combinatorial $d$-ball}) if $K$ (more precisely, the geometric realization of~$K$) is PL homeomorphic to the standard sphere~$S^d$ (respectively, to the standard disk~$D^d$). A simplicial complex~$K$ is called a \textit{combinatorial $d$-manifold} if the link of every vertex of~$K$ is a combinatorial $(d-1)$-sphere; then the link of any $k$-simplex of~$K$ is a combinatorial $(d-k-1)$-sphere, see~\cite[Corollary~1 of Lemma~9]{Zee63}. For further background material on simplicial complexes, see~\cite{RoSa72} or~\cite{Zee63}.

Consider a field $\F_{27}$ of $27$ elements. The additive group $\F_{27}\cong \rC_3^3$ and the multiplicative group~$\F_{27}^{\times}\cong \rC_{13}\times \rC_2$ act on~$\F_{27}$ by additive shifts and multiplicative shifts, respectively. These two actions together yield the action on~$\F_{27}$ of the $702$-element general affine group $\GA(1,\F_{27})=\F_{27}\rtimes\F_{27}^{\times}$. We will be interested in a twice less group, namely, the $351$-element group
$$
G_{351}=\F_{27}\rtimes(\F_{27}^{\times})^2\cong \rC_3^3\rtimes \rC_{13},
$$
where $(\F_{27}^{\times})^2\subset \F_{27}^{\times}$ is the index~$2$ subgroup consisting of squares. 

Our aim is to construct $G_{351}$-invariant combinatorial manifolds on the vertex set $\F_{27}$ that will be triangulations of manifolds like the octonionic projective plane. To work with simplicial complexes on the  set $\F_{27}$ we need to number the elements of this set, that is, to choose a bijection $\varphi\colon \F_{27}\to [27]$, where we use the notation $[n]=\{1,\ldots,n\}$. We will always conveniently identify $\F_{27}=\F_3(\alpha)$, where $\alpha^3-\alpha-1 = 0$. Then $\alpha^{13} = 1$ and hence $\alpha^k$, where $k=0,\ldots,12$, are all nonzero squares in~$\F_{27}$, and $-\alpha^k$, where $k=0,\ldots,12$, are all non-squares in~$\F_{27}$. We choose the numbering~$\varphi$ as follows:
\begin{align*}
\varphi(\alpha^k)&=k+1,&k&=0,\ldots,12,\\
\varphi(-\alpha^k)&=k+14,&k&=0,\ldots,12,\\
\varphi(0)&=27.&& 
\end{align*}

It is easy to see that the group $G_{351}$ is generated by the multiplicative shift by~$\alpha$ and the additive shift by~$1$. In the chosen numbering, these two elements of~$G_{351}$ correspond to the permutations
\begin{gather*}
A = (1\ 2\ 3\ 4\ 5\ 6\ 7\ 8\ 9\ 10\ 11\ 12\ 13)(14\ 15\ 16\ 17\ 18\ 19\ 20\ 21\ 22\ 23\ 24\ 25\ 26),\\
B = (1\ 14\ 27)(2\ 4\ 10)(3\ 22\ 13)(5\ 6\ 21)(7\ 25\ 11)(8\ 19\ 18)(9\ 16\ 26)(12\ 20\ 24)(15\ 23\ 17),
\end{gather*}
respectively. So $G_{351}$ can be regarded as the subgroup of~$\rS_{27}$ generated by~$A$ and~$B$.

Now, consider the normalizer~$N(G_{351})$ of~$G_{351}$ in~$\rS_{27}$. It is easy to point out two permutations of~$\F_{27}$ that normalize~$G_{351}$, the sign reversal $x\mapsto -x$ and the Frobenius automorphism $x\mapsto x^3$. In the numbering we use they correspond to the permutations
\begin{gather*}
S =(1\ 14)(2\ 15)(3\ 16)(4\ 17)(5\ 18)(6\ 19)(7\ 20)(8\ 21)(9\ 22)(10\ 23)(11\ 24)(12\ 25)(13\ 26),\\
F=(2\ 4\ 10)(3\ 7\ 6)(5\ 13\ 11)(8\ 9\ 12)(15\ 17\ 23)(16\ 20\ 19)(18\ 26\ 24)(21\ 22\ 25),
\end{gather*} 
respectively. 

\begin{propos}\label{propos_normalizer}
The normalizer~$N(G_{351})$ is the $2106$-element group $G_{351}\rtimes (\rC_2\times \rC_3)$, where the factors~$\rC_2$ and~$\rC_3$ are generated by~$S$ and~$F$, respectively. Therefore, $$N(G_{351})/G_{351}\cong \rC_2\times \rC_3.$$
\end{propos}

The proof of this proposition will be given in Section~\ref{section_no_add_sym}.

Consider the set $\mathcal{S}$ of all simplicial complexes on the vertex set~$[27]$ that are invariant with respect to the subgroup~$G_{351}\subset \rS_{27}$.
The normalizer~$N(G_{351})$ acts on~$\mathcal{S}$ and the action of the group~$G_{351}$ on~$\mathcal{S}$  is trivial. So we obtain the action of the $6$-element quotient group $N(G_{351})/G_{351}\cong \rC_2\times \rC_3$ on~$\mathcal{S}$.

\begin{theorem}\label{theorem_main_detail}
There are exactly $24$ simplicial complexes on the  vertex set~$[27]$ that are 
 $G_{351}$-invariant  combinatorial $16$-manifolds not homeomorphic to the sphere. The group $N(G_{351})/G_{351}\cong \rC_2\times \rC_3$ acts freely on the set of these $24$ simplicial complexes; hence, they are divided into $4$ orbits of $6$ complexes so that the complexes in each orbit are isomorphic to each other. The combinatorial manifolds~$K_1$, $K_2$, $K_3$, and~$K_4$ whose $G_{351}$-orbits of $16$-simplices are listed in Tables~\mbox{\ref{table_1234}--\ref{table_4}} are representatives of these four orbits. They are vertex-transitive, satisfy properties~(1) and~(2) in Theorem~\ref{theorem_main}, and are pairwise non-isomorphic. Besides, for each~$i$, the symmetry group~$\Sym(K_i)$ is exactly~$G_{351}$.
\end{theorem}

\begin{remark}
Tables~\mbox{\ref{table_1234}--\ref{table_4}} are contained in Appendix~A. They are organized as follows. Every simplex is encoded by a row of $27$ binary digits. The $i$\textsuperscript{th} from the left digit is $1$ whenever $i$ is a vertex of the simplex, and is~$0$ otherwise.
Tables~\mbox{\ref{table_1234}--\ref{table_4}} contain not all $16$-simplices of the simplicial complexes~$K_1$,  $K_2$, $K_3$, and~$K_4$ but exactly one representative of every $G_{351}$-orbit of such simplices. We conveniently interpret every row of $27$ binary digits in Tables~\mbox{\ref{table_1234}--\ref{table_4}} as a reversed binary notation for a number. For instance, the first row $$111111111110111011100000000$$ in Table~\ref{table_1234} corresponds to the number
$$1110111011111111111_2=489472.$$
With this interpretation, each of Tables~\mbox{\ref{table_1234}--\ref{table_4}} contains the smallest representatives of the orbits; in addition, each table uses ordering from smallest to largest.
\end{remark}

\begin{remark}
Attached to this paper are four text files \texttt{K1.dat}, \texttt{K2.dat}, \texttt{K3.dat}, and \texttt{K4.dat} containing the lists of (the smallest) representatives for the $G_{351}$-orbits of $16$-simplices of $K_1$, $K_2$, $K_3$, and~$K_4$, respectively, in the same format as in Tables~\mbox{\ref{table_1234}--\ref{table_4}}. The first row of each of the files contains the number $286$ of the orbits of $16$-simplices. These files fit the format for simplicial complexes used in our computer programs, see~\cite{Gai-prog}.
\end{remark}

\begin{remark}
Note that the numbering of the four combinatorial manifolds~$K_1$, $K_2$, $K_3$, and~$K_4$ is chosen so that to stress the following interesting property: If a $16$-simplex~$\sigma$ belongs to~$K_i$ and~$K_j$ for $i<j$, then $\sigma$ also belongs to all intermediate complexes~$K_k$ with $i<k<j$. We see that 
\begin{itemize}
\item $K_2$ is obtained from~$K_1$ by replacing certain $62$ \ $G_{351}$-orbits of $16$-simplices with other $62$ \ $G_{351}$-orbits of $16$-simplices,
\item $K_3$ is obtained from~$K_2$ by replacing certain $27$ \ $G_{351}$-orbits of $16$-simplices with other $27$ \ $G_{351}$-orbits of $16$-simplices,
\item $K_4$ is obtained from~$K_3$ by replacing certain $115$ \ $G_{351}$-orbits of $16$-simplices with other $115$ \ $G_{351}$-orbits of $16$-simplices.
\end{itemize}
The relationship between~$K_2$ and~$K_3$ is clear, see Section~\ref{section_moves}. Namely, $K_3$ is obtained from~$K_2$ by replacing $351$ subcomplexes of the form
$$
(\Delta_1*\partial\Delta_2)\cup(\Delta_2*\partial\Delta_3)\cup(\Delta_3*\partial\Delta_1),
$$
where $\dim\Delta_1=\dim\Delta_2=\dim\Delta_3=8$, with the corresponding subcomplexes
$$
(\partial\Delta_1*\Delta_2)\cup(\partial\Delta_2*\Delta_3)\cup(\partial\Delta_3*\Delta_1).
$$
In particular, this implies that $K_2$ and~$K_3$ are PL homeomorphic.

The relationships between~$K_1$ and~$K_2$ and between~$K_3$ and~$K_4$ are less clear, and seem to be interesting to study further. At the moment, the author does not know whether it is true that $K_1$ or~$K_4$ is PL homeomorphic to~$K_2$ and~$K_3$.

If we consider all the $24$ combinatorial manifolds $qK_i$ from Theorem~\ref{theorem_main_detail}, where $q$ runs over the elements of the group~$\rC_2\times \rC_3$ generated by~$S$ and~$F$, then it is interesting to note that the four combinatorial manifolds~$K_1$, $K_2$, $K_3$, and~$K_4$ form a cluster in the sense that their pairwise intersections are large enough, namely, consist of at least~$112$ \ $G_{351}$-orbits of $16$-simplices. On the contrary, each intersection $q_1K_i\cap q_2 K_j$ with $q_1\ne q_2$ contains a sufficiently small number of $16$-simplices, namely, at most~$12$ \ $G_{351}$-orbits of $16$-simplices. A precise numbers of $G_{351}$-orbits of common $16$-simplices are given in Table~\ref{table_intersections}. (Certainly, the intersections~$q_1K_i\cap q_2 K_j$ and~$qq_1K_i\cap qq_2 K_j$ are isomorphic for each $q\in \rC_2\times \rC_3$, so the information in this table is enough to restore the numbers of orbits of $16$-simplices in all intersections~$q_1K_i\cap q_2 K_j$.)
 \begin{table}
\caption{Numbers of common orbits of $16$-simplices}\label{table_intersections}
\begin{tabular}{|c|cccccccc|}
\hline
$\phantom{\strut^2}$ & $K_1$ & $K_2$ & $K_3$ & $K_4$ & 
$SK_1$ & $SK_2$ & $SK_3$ & $SK_4$ \\
\hline
$K_1$  & 286 & 224 & 201 & 112 & 12 & 8 & 8 & 4 \\
$K_2$  & 224 & 286 & 259 & 149 & 8 & 5 & 5 & 3 \\
$K_3$  & 201 & 259 & 286 & 171 & 8 & 5 & 5 & 3 \\
$K_4$  & 112 & 149 & 171 & 286 & 4 & 3 & 3 & 3 \\
\hline
\hline
$\phantom{\strut^2}$ & $FK_1$ & $FK_2$ & $FK_3$ & $FK_4$ &
 $SFK_1$ & $SFK_2$ & $SFK_3$ & $SFK_4$ \\ 
 \hline
$K_1$ & 9 & 6 & 6 & 5 & 2 & 0 & 0 & 0 \\
$K_2$ & 7 & 6 & 6 & 5 & 5 & 2 & 1 & 1 \\
$K_3$ & 7 & 6 & 6 & 4 & 5 & 4 & 3 & 3 \\
$K_4$ & 9 & 7 & 7 & 7 & 6 & 6 & 5 & 6 \\
\hline
\end{tabular}
\end{table}
\end{remark}

\section{Scheme of proof of Theorem~\ref{theorem_main_detail}}\label{section_scheme}

The proof will consist of 5 steps.

\textsl{Step 1.} We will give an algorithm that, being given 
\begin{itemize}
\item a number $m$ and a finite subgroup $G\subset \rS_m$,
\item a dimension~$d$,
\item a positive integer~$N$,
\end{itemize}
produces the list of all  simplicial complexes~$K$ on the vertex set~$[m]$ satisfying the following properties:
\begin{itemize}
\item $K$ is a $G$-invariant weak $d$-pseudomanifold,
\item for any two simplices $\sigma,\tau\in K$,  $\sigma\cup\tau$ is not the whole set~$[m]$,
\item the number of $d$-simplices of~$K$ is at least $N$.
\end{itemize}

We are interested in the case
\begin{equation}\label{eq_main_case_init}
m=27,\qquad d =16,\qquad G=G_{351},\qquad N=100386.
\end{equation}
So our aim is an algorithm that, when implemented on a computer, will complete in reasonable amount of time in this case. We will describe such an algorithm in Section~\ref{section_find}. In the case~\eqref{eq_main_case_init}, a computer implementation of this algorithm yields the following result.

\begin{propos}\label{propos_24_pseudo}
There are exactly $24$ simplicial complexes~$K$ on the vertex set~$[27]$ with the following properties: 
\begin{itemize}
\item $K$ is a $G_{351}$-invariant weak $16$-pseudomanifold,
\item for any two simplices $\sigma,\tau\in K$,  $\sigma\cup\tau$ is not the whole set~$[27]$,
\item $K$ has at least $100386$ top-dimensional simplices.
\end{itemize}
Four of these $24$ complexes are the complexes~$K_1$, $K_2$, $K_3$, and~$K_4$ whose orbits of simplices are listed in Tables~\mbox{\ref{table_1234}--\ref{table_4}}; the other $20$ complexes are obtained from those four by the action of the group~$\rC_2\times \rC_3$ generated by~$S$ and~$F$.
\end{propos}

At the moment, we do not know yet that the simplicial complexes~$K_1$, $K_2$, $K_3$, and~$K_4$ provided by our algorithm are combinatorial manifolds that are not homeomorphic to the sphere. Indeed, they could be not combinatorial manifolds or they could be homeomorphic to~$S^{16}$. However, recall that every $G_{351}$-invariant   combinatorial $16$-manifold not homeomorphic to~$S^{16}$ (if exists) must satisfy the complementarity condition and have exactly $100386$ top-dimensional simplices. So all such combinatorial manifolds are in the list of the $24$ pseudomanifolds from Proposition~\ref{propos_24_pseudo}. Therefore, up to isomorphism they are in the list $K_1$, $K_2$, $K_3$, $K_4$. So these four simplicial complexes constitute a list of candidates that we need  to study further. 

For each $i$, a direct computation shows that the face numbers~$f_k(K_i)$ are exactly as in Table~\ref{table_f} and hence $\chi(K_i)=3$. Therefore, $K_i$ is not homeomorphic to~$S^{16}$.
 
\textsl{Step 2.} Next, we need to prove the following proposition.

\begin{propos}\label{propos_combman}
The weak pseudomanifolds~$K_1$, $K_2$, $K_3$, and~$K_4$ are combinatorial manifolds.
\end{propos}

To prove that $K_i$ is a combinatorial manifold, we need to check that the links of its vertices are PL homeomorphic to~$S^{15}$. There are several standard methods for doing this, mostly implemented in the system \texttt{GAP}, such as the program \texttt{BISTELLAR} by Bj\"orner and Lutz, see~\cite{BjLu00}, and several programs based on discrete Morse functions, see~\cite{BeLu14} and references therein. Nevertheless, most of these methods are useless in our situation. The matter is that the link of a vertex of each~$K_i$ has a very large size, namely, it has dimension~$15$, $26$ vertices, $63206$ top-dimensional simplices, and $29193216$ (non-empty) simplices of all dimensions. In addition, this link has few symmetries, namely, its symmetry group is only~$\rC_{13}$. In such a situation, it is hopeless to try to simplify the complex by bistellar moves or any other combinatorial moves. Also, standard methods for constructing discrete Morse functions are too slow in our situation, since they require working with simplices of all dimensions whose number is really huge. Instead, we use an approach based on the notion of nonevasiveness of simplicial complexes introduced by Kahn, Saks, and Sturtevant~\cite{KSS84}. Its advantage is that we can check the nonevasiveness of a simplicial complex by working only with its vertices and maximal simplices, without referring directly to simplices of intermediate dimensions. Conceptually, our approach is close to that of Engstr\"om~\cite{Eng09}, though we prefer to prove the nonevasiveness of complexes directly from the definition rather than to construct any kind of discrete Morse functions.

An explicit algorithm that, when implemented on a computer, checks that $K_1$, $K_2$, $K_3$, and~$K_4$ are combinatorial manifolds is described in Section~\ref{section_check_manifold}. Since $K_1$, $K_2$, $K_3$, and~$K_4$  are not homeomorphic to~$S^{16}$, it now follows from a result of Brehm and K\"uhnel~\cite{BrKu87} that they are manifolds like the octonionic projective plane and satisfy conditions~(a)--(e) in the introduction.

\textsl{Step 3.} Now, let us show that the simplicial complexes $K_1$, $K_2$, $K_3$, and~$K_4$ are pairwise non-isomorphic. There are several ways to do this. Probably, the easiest one is as follows. Each simplicial complex~$K_i$ has $f_{14}=3309696$ simplices of dimension~$14$.  For each $14$-simplex $\rho\in K_i$, we can compute the number~$s(\rho)$ of $16$-simplices $\sigma\in K_i$ such that $\rho\subset \sigma$. It turns out that the number $s(\rho)$ varies from $3$ to~$9$ for each of the complexes~$K_i$. Now, we can study the distribution of the function~$s(\rho)$, that is, compute the number of simplices $\rho$ with each given value of~$s(\rho)$. The result of the computation is shown in Table~\ref{table_distribution}. We see that the distributions in the four columns are pairwise different and hence the simplicial complexes~$K_1$, $K_2$, $K_3$, and~$K_4$ are pairwise non-isomorphic.

\begin{table}
\caption{Numbers of $14$-simplices~$\rho$ with the given number~$s(\rho)$ of adjacent $16$-simplices}\label{table_distribution}
\begin{tabular}{|c|c|c|c|c|}
\hline
$s(\rho)$ & $K_1$ & $K_2$ & $K_3$ & $K_4$\\
\hline
3  & 849771 & 953316 & 940446 & 868140\\
4  & 1509651  & 1364688 & 1363986 & 1447524\\ 
5  & 697788 & 690066 & 724815 & 764829\\
6  & 201942 & 237978 & 220662 & 180297\\
7  & 40716 & 55107 & 51597 &  40716 \\
8  & 9477 &  7020 & 7722 & 6669\\
9  & 351 & 1521 & 468 & 1521\\
\hline
\end{tabular}
\end{table}

\textsl{Step 4.} Since each $K_i$ is $G_{351}$-invariant and $G_{351}$ acts transitively on~$\F_{27}$, we see that $K_i$ is vertex-transitive. Let us prove that $K_i$  satisfies conditions~(1) and~(2) from Theorem~\ref{theorem_main}. 

Recall that the group~$G_{351}$ coincides with the group of all affine transformations of~$\F_{27}$ of the form $x\mapsto ax+b$, where $a\in(\F_{27}^{\times})^2$ and $b\in \F_{27}$. We need to prove that $G_{351}$ acts transitively on the set of undirected edges, that is, on the set of two-element subsets of~$\F_{27}$. Consider an arbitrary subset $\{u,v\}\subset \F_{27}$, where $u\ne v$. Swapping $u$ and~$v$, we may achieve that $v-u$ is a square. Then the transformation $x\mapsto (v-u)x+u$ belongs to~$G_{351}$ and takes $\{0,1\}$ to~$\{u,v\}$, hence the required transitivity. 

To show that $G_{351}$ acts freely on the set of $16$-simplices of every~$K_i$, we prove a more general assertion.

\begin{propos}\label{propos_free}
The group~$G_{351}$ acts freely on the set of all $17$-element subsets of~$\F_{27}$.
\end{propos}

\begin{proof}
Any nontrivial transformation $g\colon x\mapsto ax+b$ in~$G_{351}$ can be written in the form
$$
x\mapsto
\left\{
\begin{aligned}
&b\left(b^{-1}x+1\right)&&\text{if}\ a=1,\,b\in (\F_{27}^{\times})^2, \\
&-b\left(-b^{-1}x-1\right)&&\text{if}\ a=1,\,b\notin (\F_{27}^{\times})^2,\\
&a(x+c)-c&&\text{if}\ a\ne 1,
\end{aligned}
\right.
$$
where in the third case $c=(a-1)^{-1}b$. Since $a$ is a square, in the third case we have $a=\alpha^k$ with $k\in\{1,\ldots,12\}$. Hence, $g$ is conjugate in~$G_{351}$ to one of the elements~$B^{\pm 1}$ or~$A^k$, where $k=1,\ldots, 12$. Therefore, the permutation~$g\in \rS_{27}$ has the same cycle type as either~$A$ or~$B$. Now,  a permutation of either of these cycle types cannot stabilize a $17$-element subset, since no union of its cycles has the sum of lengths~$17$.
\end{proof}

\textsl{Step 5.} Finally, we need to prove the following proposition.

\begin{propos}\label{propos_SymK}
For every $i=1,2,3,4$, we have $\Sym(K_i)=G_{351}$.
\end{propos} 

By construction, $K_i$ is invariant under the action of~$G_{351}$. So to prove Proposition~\ref{propos_SymK}, we need to check that $K_i$ has no additional symmetries. This will be done in Section~\ref{section_no_add_sym}.

\smallskip

Thus, we have proved Theorem~\ref{theorem_main_detail} modulo Propositions~\ref{propos_24_pseudo}, \ref{propos_combman}, and \ref{propos_SymK}. These three propositions will be proved in Sections~\ref{section_find}, \ref{section_check_manifold} and~\ref{section_no_add_sym}, respectively.

\section{Algorithm for finding triangulations}\label{section_find}

In this section we describe an algorithm that solves the problem from Step~$1$ in the scheme of the proof of Theorem~\ref{theorem_main_detail}. Recall that the algorithm should produce the list of all weak pseudomanifolds~$K$ of the given dimension~$d$ on the given set of vertices~$[m]$ such that $K$ is invariant with respect to the given subgroup~$G\subset \rS_m$, contains at least the given number~$N$ of $d$-simplices, and satisfies the following half of the complementarity condition:
\begin{itemize}
\item for any two simplices $\sigma,\tau\in K$, the union $\sigma\cup\tau$ is not the whole set~$[m]$.
\end{itemize}

We are mostly interested in the case
\begin{equation}\label{eq_main_case}
m=27,\qquad d =16,\qquad G=G_{351},\qquad N=100386,
\end{equation}
and our aim is an algorithm that, when implemented on a computer, will complete in reasonable amount of time in this case. 

We conveniently reformulate our problem as follows. Let $\CO$ be the set of all $G$-orbits of $(d+1)$-element subsets of~$[m]$. Each desired simplicial complex~$K$ is uniquely characterized by the set of its $d$-simplices. Moreover, this set must be $G$-invariant, so each orbit $\fo\in\CO$ is either completely included in~$K$, or none of subsets in~$\fo$ is included in~$K$. So our goal is equivalently reformulated in the following way:

\smallskip

\textit{Produce the list of all subsets $\CX\subset \CO$ satisfying the following properties:
\begin{itemize}
\item[(i)] For each $d$-element subset $\rho\subset [m]$, the number of $(d+1)$-element subsets $\sigma$ that contain~$\rho$ and belong to one of the orbits in~$\CX$ is either $0$ or~$2$.
\item[(ii)]  For any two subsets  $\sigma$ and~$\tau$ each of which belongs to one of the orbits in~$\CX$,  $\sigma\cup\tau\ne [m]$.
\item[(iii)] The total number of subsets in all orbits belonging to~$\CX$ is at least~$N$.
\end{itemize}
}

The cardinality~$|\CO|$ is approximately $\frac{1}{|G|}\binom{m}{d+1}$. In the most interesting  case for us, \eqref{eq_main_case}, the action of~$G_{351}$ on the set of $17$-element subsets of~$[27]$ is free by Proposition~\ref{propos_free}. So in this case, we have an exact equality
\begin{gather*}
|\CO|=\frac{1}{|G|}\binom{m}{d+1}=\frac{1}{351}\cdot\binom{27}{17}=24035.
\end{gather*}

\begin{defin}
We will say that an orbit $\fo\in\CO$ is \textit{admissible} if it satisfies neither of the following two conditions:
\begin{itemize}
\item there exist subsets $\sigma,\tau\in\fo$ such that $\sigma\cup\tau=[m]$,
\item there exists a $d$-element subset $\rho\subset [m]$ that is contained in at least $3$ subsets in~$\fo$.
\end{itemize} 
\end{defin}
We denote by $\COa$   the subset of~$\CO$ consisting of all admissible orbits.

Only admissible orbits can enter a set~$\CX$  satisfying properties~(i) and~(ii). So we start with listing all the admissible orbits. This operation is very fast, since it requires only a simple enumeration of orbits in~$\CO$. In the case~\eqref{eq_main_case}, we have $18546$ admissible orbits. Technically, each subset is encoded by a binary number as in Tables~\ref{table_1234}--\ref{table_4} and  listing the admissible orbits means listing of the smallest representatives of them.

Further, the algorithm consists of two stages.

\textsl{Stage 1. Initial prohibition of some pairs of orbits and listing of all adjacency groups.} We are interested in finding pairs~$\{\fa,\fb\}$ of (different) admissible orbits that are \textit{prohibited} in the sense that they cannot enter~$\CX$ simultaneously. There are two reasons for such prohibition.
The first one is as follows. If there are subsets $\sigma\in\fa$ and~$\tau\in\fb$ such that $\sigma\cup\tau=[m]$, then the  pair~$\{\fa,\fb\}$ is prohibited because its presence in~$\CX$ would violate condition~(ii). All prohibited pairs of orbits of this kind are found by a direct enumeration over all pairs of admissible orbits. 

To describe the second reason for prohibition of a pair of orbits, we need to introduce some terminology. Suppose that $\fo\in \COa$ and $\rho\subset [m]$ is a $d$-element subset. Since $\fo$ is admissible, the number of subsets $\sigma \in\fo$ that contain~$\rho$ does not exceed~$2$. We will say that  an orbit $\fo$ is \textit{once adjacent} to~$\rho$ if $\fo$ contains exactly one subset~$\sigma$ such that $\sigma\supset\rho$ and is \textit{twice adjacent} to~$\rho$ if $\fo$ contains two subsets~$\sigma$ such that $\sigma\supset\rho$.  

Now, if $\fa$ is twice adjacent to a $d$-element subset~$\rho$ and $\fb$ is either once or twice adjacent to the same~$\rho$, then the pair~$\{\fa,\fb\}$ is prohibited because its presence in~$\CX$ would violate condition~(i). All prohibited pairs of orbits of this kind are found by a direct enumeration over~$\rho$.  The adjacency of an orbit in~$\COa$ to~$\rho$ will not change if $\rho$ is replaced with another $d$-element subset $\rho'$ in the $G$-orbit of~$\rho$. So, in fact, the enumeration is not over all $d$-element subsets but only over all their $G$-orbits. Moreover, we can restrict ourselves to only those~$\rho$ that are contained in at least one of the smallest representatives of admissible orbits.

In addition, at the same stage, we find all \textit{adjacency groups}~$\CA_{\rho}$, where $\rho$ runs over $G$-orbits of $d$-element subsets of~$[m]$. By definition, $\CA_{\rho}$ is the subset of~$\COa$ consisting of all orbits that are once adjacent to~$\rho$. The role of these groups in the further algorithm is as follows. By condition~(i), the set $\CX$ must contain either no or exactly two orbits in each adjacency groups. Linking adjacency groups to subsets~$\rho$ is no longer important to us, we need only the list of all non-empty adjacency groups $\CA_1,\ldots,\CA_q$. In the case~\eqref{eq_main_case}, we obtain $36059$ non-empty adjacency groups.
  
\begin{remark}
In fact, in our implementation of the algorithm we never check that the generated adjacency groups are pairwise different. It may happen that different orbits of subsets~$\rho$ lead to identical adjacency groups. We prefer to work with a list of adjacency groups that can contain repetitions, rather than  eliminating repetitions.
\end{remark}

From now on, we may forget about the nature of the elements of~$\COa$. Indeed, our goal is reduced to the following combinatorial problem:

\smallskip
\textit{We have a finite set~$\COa$. With each $\fo\in\COa$ is associated a positive integer~$|\fo|$ called its size. Some pairs of elements of~$\COa$ are prohibited. Besides, we have a list $\CA_1,\ldots,\CA_q$ of subsets of~$\COa$ called adjacency groups. Our aim is to list all subsets $\CX\subset \COa$ such that
\begin{itemize}
\item each intersection $\CX\cap\CA_i$ is either empty or consists of exactly two elements,
\item $\CX$ is not allowed to contain simultaneously both elements of any prohibited pair,
\item $\sum_{\fo\in\CX}|\fo|\ge N$.
\end{itemize}
}

\textsl{Stage 2. Selection of orbits.} At this stage we solve the above combinatorial problem. We will not use the nature of elements of~$\COa$ any more. However, for convenience we will continue to call them orbits.

Suppose that, at some moment, we have the following situation:
\begin{itemize}
\item Some orbits in~$\COa$ are \textit{taken}, which means that we have decided that they are included in~$\CX$. Some orbits in~$\COa$ are \textit{removed}, which means that we have decided that they are not included in~$\CX$. All other orbits are called \textit{indeterminate}. 
\item Some unordered pairs of indeterminate orbits are \textit{prohibited}, which means that they are not allowed to enter~$\CX$ simultaneously.
\item Some ordered pairs of indeterminate orbits form \textit{requirements}. A requirement from $\fa$ to~$\fb$ (denoted~$\fa\to\fb$) means that we must include~$\fb$ to~$\CX$ whenever we include~$\fa$.
\end{itemize} 
These data should satisfy natural logical conditions, namely (a) if $\fa\to\fb$, then $\{\fa,\fb\}$ is not prohibited, (b) if $\fa\to\fb$ and $\fb\to\fc$, then $\fa\to\fc$, and (c) if $\fa\to\fb$ and $\{\fb,\fc\}$ is prohibited, then $\{\fa,\fc\}$ is also prohibited.

Note that the initial situation is of this type. Initially all orbits are indeterminate, we have several prohibited pars and no requirements. 

Having a situation of the described type, we want to improve our knowledge by (1)  increasing the numbers of taken and removed orbits or (2) increasing the numbers of prohibited pairs and requirements without changing taken and removed orbits. To do this we study the adjacency groups.

For an adjacency group~$\CA$, we may have several situations:

\smallskip

\textsl{Case 1. More than two orbits in~$\CA$ are already taken.} This means that our task is inconsistent and we should stop seeking for~$\CX$.

\smallskip

\textsl{Case 2. Exactly two orbits in~$ \CA$ are already taken.} This means that all other orbits in~$ \CA$ cannot enter~$\CX$, so we should remove them. Note that whenever we remove an orbit~$\fa$, we always simultaneously remove all orbits~$\fb$ (not necessarily belonging to~$ \CA$) such that $\fb\to\fa$. After doing this, we may forget about the adjacency group~$ \CA$. It will not give us any more information in the future.

\smallskip

\textsl{Case 3. Exactly one orbit in~$ \CA$ is already taken.} Let $\fa$ be this taken orbit, and let $\fb_1,\ldots,\fb_k$ be all indeterminate orbits in~$ \CA$. First, if $k=0$, then our task is inconsistent again, and we stop. Second, if $k=1$, then the unique indeterminate orbit $\fb$ should be taken. Note that whenever we take an orbit~$\fb$, we always simultaneously take all orbits~$\fc$ such that $\fb\to\fc$ and remove all orbits~$\fd$ such that the pair~$\{\fb,\fd\}$ was prohibited. Third, suppose that $k\ge 2$. In this case we cannot take or remove any of the orbits $\fb_1,\ldots,\fb_k$. However, we can guarantee that any two of these orbits cannot enter~$\CX$ simultaneously. Hence, we should prohibit all pairs $\{\fb_i,\fb_j\}$ with $i\ne j$. Note that whenever we prohibit a pair $\{\fb,\fb'\}$, we always simultaneously prohibit all pairs $\{\fc,\fc'\}$ such that $\fc\to\fb$ (or $\fc=\fb$) and $\fc'\to\fb'$ (or $\fc'=\fb'$). Moreover, if on this way we are faced with the need to prohibit a pair $\{\fc,\fc\}$ (say, if $\fc\to\fb$ and simultaneously $\fc\to \fb'$), then we should remove~$\fc$.

\smallskip

\textsl{Case 4. No orbit in~$ \CA$ is taken.} Consider 3 subcases:

\smallskip

\textsl{Subcase 4a. There are three distinct indeterminate orbits $\fa$, $\fb$, and~$\fc$ in~$ \CA$ such that $\fa\to \fb$ and $\fa\to\fc$.} If $\fa$ were in~$\CX$, then both orbits~$\fb$ and~$\fc$ should also be in~$\CX$, and the intersection $\CX\cap \CA$ would contain more than two orbits, which is impossible. Therefore, we should remove the orbit~$\fa$ (and all orbits that require it).

\smallskip

\textsl{Subcase 4b. We are not in Subcase~4a but there are two distinct indeterminate orbits $\fa$ and~$\fb$ in~$ \CA$ such that $\fa\to \fb$.} If $\fa$ were in~$\CX$, then~$\fb$ would also be in~$\CX$ and hence no other orbit in~$ \CA$ could be in~$\CX$. Hence, we should prohibit all pairs $\{\fa,\fc\}$, where $\fc$ is an indeterminate orbit in~$ \CA$ different from~$\fa$ and~$\fb$. (Again, together with any such pair of orbits we prohibit all dependent pairs as in Case~3.)

\smallskip

\textsl{Subcase 4c. There is an indeterminate orbit~$\fa$ in~$ \CA$ such that the pairs $\{\fa,\fb\}$ are prohibited for all other  indeterminate orbits~$\fb$ in~$ \CA$.} (In particular, this subcase occurs if there is only one indeterminate orbit in~$ \CA$.) Then we should remove~$\fa$ and all orbits that require it.

\smallskip

\textsl{Subcase 4d. There is an indeterminate orbit~$\fa$ in~$ \CA$ such that the pairs $\{\fa,\fb\}$ are prohibited for all but one other indeterminate orbits~$\fb$ in~$ \CA$.} (In particular, this subcase occurs if there are exactly two indeterminate orbits in~$ \CA$.) Let $\fc$ be the only indeterminate orbit in~$ \CA$ such that $\fc\ne\fa$ and the pair $\{\fa,\fc\}$ is not prohibited. Then $\fc$ will belong to~$\CX$ whenever $\fa$ belongs to~$\CX$. So we should add a requirement $\fa\to \fc$. Note that whenever we add such a requirement, we simultaneously add all requirements $\fa'\to\fc'$, where $\fa'\to\fa$ (or $\fa'=\fa$) and $\fc\to\fc'$ (or $\fc=\fc'$). Besides, we simultaneously prohibit all pairs~$\{\fa',\fd\}$, where $\fa'\to\fa$ (or $\fa'=\fa$) and the pair $\{\fc,\fd\}$ was prohibited. Also, as in Case~3, if on this way we are faced with the need to prohibit a pair $\{\fd,\fd\}$, then we  remove~$\fd$.

\smallskip

When we have examined an adjacency group~$\CA$ (see Cases~1--4 above), it makes no sense for us to return to the same adjacency group~$\CA$ until a new indeterminate orbit in~$\CA$ is taken or removed or a new pair of orbits in~$\CA$ is prohibited or a new requirement between orbits in~$\CA$ is added.

So the algorithm is organized as follows. At every moment we have a waiting list of adjacency groups that are to be examined. Initially all adjacency groups are in the list. Adjacency groups in the list are examined one by one as described above. After an adjacency group is examined, we expel it from the list. However, every time a new indeterminate orbit~$\fa$ becomes either taken or removed, we put back to the waiting list all adjacency groups containing~$\fa$. Similarly, every time a new pair of orbits~$\{\fa,\fb\}$ is prohibited or a new requirement $\fa\to\fb$ is created, we put back to the waiting list all adjacency groups containing both~$\fa$ and~$\fb$. This process stops when the waiting list becomes empty. If at this moment all orbits are either taken or removed, then a set~$\CX$ with the required properties is found. If not, then our algorithm branches. We choose an indeterminate orbit~$\fo$ and try to first take it and then to remove it. In either case, we start the same process again.

When the algorithm branches, we have the freedom to choose for~$\fo$ any of the indeterminate orbits. It turns out that this choice significantly affects the running time  of the algorithm. For an indeterminate orbit~$\fa$, we denote by~$p(\fa)$ the number of prohibited pairs~$\{\fa,\fb\}$ and by~$r(\fa)$ the number of requirements $\fa\to\fb$, where in both cases $\fb$ runs over all indeterminate orbits. It seems advantageous to choose for~$\fo$ an orbit with large~$p(\fo)$ and~$r(\fo)$, since taking of~$\fo$ will immediately lead to removing and taking of many orbits. Moreover, the intuition is that taken orbits are more important than removed ones, since they give more potential to analyze adjacency groups. So we use the following heuristic. We choose for~$\fo$ the orbit (any if several) with the largest value of $p(\fo)+10r(\fo)$. Practical calculations show that such a heuristic significantly improves the speed of the algorithm.

Until now, we have never used the restriction 
\begin{equation}\label{eq_restr}
\sum_{\fo\in\CX}|\fo|\ge N.
\end{equation}
Nevertheless, it plays an important role in bounding the running time of the algorithm and the size of the output. When we implemented the described algorithm in the most interesting case for us,  $m=27$, $d=16$, $G=G_{351}$, without any restriction of the form~\eqref{eq_restr}, it began to give out an enormous list of weak pseudomanifolds satisfying conditions~(i) and~(ii) with the numbers of $16$-simplices mostly in the range from $15000$ to~$30000$, and did not stop in a reasonable time.

So we need to make the following improvements to the algorithm:
\begin{enumerate}
\item Throughout the whole algorithm we look after the sum~$S$ of sizes $|\fo|$ of all orbits that are either taken or indeterminate, and stop considering any branch that leads us to a situation with this sum less than~$N$.  
\item When the waiting list of adjacency groups becomes empty, before starting branching the algorithm, we may try to resolve the situation in the following way. For each indeterminate orbit~$\fa$, we compute the sum~$P_{\fa}$ of the sizes of all indeterminate orbits~$\fb$ such that the pair $\{\fa,\fb\}$ is prohibited. Then we remove all indeterminate orbits~$\fa$ with $S-P_{\fa}<N$. (The reason for doing this is that, if $\fa$ entered~$\CX$, the maximal possible sum of sizes of orbits would become less than~$N$.) 
\end{enumerate}

Procedure~(2) is not always helpful. The matter is that the computation of the sums~$P_{\fa}$ requires enumeration over pairs of indeterminate orbits. When the number~$M$ of indeterminate orbits is large, this is a rather time-consuming procedure. Besides, when $M$ is large enough, then the sum~$S$ is much larger than~$N$ and it is unlikely that we will find at least one orbit~$\fa$ with $S-P_{\fa}<N$. So, most probably, the procedure will not be effective in this case. On the other hand, the procedure can be very useful in the case of a sufficiently small number of indeterminate orbits, that is, when the sum~$S$ is close to~$N$. (Typically, there are very few taken orbits, see Remark~\ref{remark_bound_level} below, so their contribution to~$S$ is negligible.)

We use the following heuristic. If at the moment when the waiting list of adjacency groups becomes empty, the inequality
\begin{equation}\label{eq_heuristic_bound}
M\le \frac{5N}{|G|}\,
\end{equation}
(which is approximately equivalent to $S\le 5N$) is satisfied, then we apply Procedure~(2) and start branching the algorithm only if Procedure~(2) is useless, that is, $S-P_{\fa}\ge N$ for all indeterminate orbits~$\fa$. However, if  inequality~\eqref{eq_heuristic_bound} is not satisfied, then we do not apply Procedure~(2) and start branching the algorithm immediately.

Now, the description of the algorithm is completely finished. Being implemented in the case~\eqref{eq_main_case}, this algorithm yields Proposition~\ref{propos_24_pseudo}.

\begin{remark}\label{remark_bound_level}
At first glance, it may seem that our algorithm will branch in an uncontrolled way. However, this does not happen. The matter is that if the group~$G$ is large enough, then we get a lot of prohibited pairs from the very beginning. Therefore, every time an orbit is taken, a lot of other orbits become removed immediately. This leads to the fact that branches of the algorithm in which we choose to take an orbit several times end very quickly. We conveniently say that our algorithm is on \textit{level} $k$ if $k-1$ choices in favor of taking an orbit have been made so far.  (The number of choices in favor of removing an orbit is unimportant.) In the case~\eqref{eq_main_case} the algorithm never goes beyond level~$5$, and spends most of its time at levels~$1$ and~$2$.
\end{remark}

\begin{remark}\label{remark_implement}
We have implemented the algorithm as a C++ program on a usual PC without any attempt to make parallel computing. (Note, however, that the algorithm can be very easily parallelized if necessary.) The computation was performed on one processor core of clock frequency 1.8 GHz. Under such conditions the running time of the program in the case~\eqref{eq_main_case} was about 7.5 hours.
\end{remark}

\begin{remark}
Being implemented in the case
$$
m=15,\qquad d=8,\qquad G=\rA_5\subset\rS_{15},\qquad N=490,
$$
our algorithm recovers the Brehm--K\"uhnel triangulation~$\HP^2_{15}$. Here the embedding of~$G=\rA_5$ into~$\rS_{15}$ is provided by the transitive action of~$\rA_5\cong\SL(2,\F_4)$ on the $15$ nonzero vectors in~$\F_4^2$. Alternatively, the same  embedding up to conjugation is given by the sequence of homomorphisms
$$
\rA_5\subset\rS_5\subset\rS_6\xrightarrow{\varphi}\rS_6\xrightarrow{\lambda}\rS_{15},
$$
where $\varphi$ is an outer automorphism of~$\rS_6$ and~$\lambda$ is an embedding provided by the action of~$\rS_6$ on the $15$ two-element subsets of the set~$\{1,\ldots,6\}$. In this case, the program runs very quickly  (in less than a second) and produces a list of $6$ isomorphic copies of the triangulation~$\HP^2_{15}$, which corresponds to the fact that the subgroup~$G\cong\rA_5$ has index~$6$ in its normalizer~$N(G)$ in~$\rS_{15}$. (If we realize $G$ as the group~$\SL(2,\F_4)$ acting on the $15$ nonzero vectors in~$\F_4^2$, then it is not hard to check that~$N(G)=\GL(2,\F_4)\rtimes \rC_2$, where the factor~$\rC_2$ corresponds to the Frobenius automorphism of~$\F_4$.)

Note also that in the four-dimensional case the algorithm works fast and recovers the K\"uhnel triangulation~$\CP^2_9$ even without specifying the symmetry group, that is, for the initial data
$$
m=9,\qquad d=4,\qquad G=1,\qquad N=36.
$$
\end{remark}

\section{Checking that $K_1$, $K_2$, $K_3$, and~$K_4$ are combinatorial manifolds}\label{section_check_manifold}

Our approach will be based on the properties of \textit{collapsibility} and  \textit{nonevasiveness} for simplicial complexes. Recall the definitions of these properties (see~\cite{KSS84}). 

First, introduce some notation. Suppose that $K$ is a simplicial complex on the vertex set~$V$. For a simplex~$\sigma\in K$, its \textit{link} and \textit{contrastar} are the simplicial complexes
\begin{align*}
\link(\sigma,K)&=\bigl\{\tau\subseteq V\setminus\sigma \colon \sigma\cup\tau\in K \bigr\},\\
\cost(\sigma,K)&=\bigl\{\tau\subseteq V\setminus\sigma \colon \tau\in K  \bigr\},
\end{align*}
respectively. For a subset~$U\subset V$, the \textit{full subcomplex} of~$K$ spanned by~$U$ is the complex
$$
K_U=\bigl\{\tau\subseteq U \colon \tau\in K  \bigr\}.
$$
Then $\cost(\sigma,K)=K_{V\setminus\sigma}$.

\begin{defin}[Collapsibility]
Suppose that $\sigma$ is a maximal (w.\,r.\,t. inclusion) simplex of a simplicial complex~$K$. A codimension $1$ face $\tau\subset\sigma$ is called a \textit{free facet} if it is contained in no other simplex of~$K$. An  \textit{elementary collapse} is the process of removing from~$K$ some pair $(\sigma,\tau)$, where $\sigma$ is a maximal simplex of~$K$ and $\tau$ is a free facet of~$\sigma$. We say that a simplicial complex~$K$ \textit{collapses} to a subcomplex~$L$ and write $K\searrow L$ if $L$ can be obtained from~$K$ by a finite sequence of elementary collapses. A simplicial complex is called \textit{collapsible} if it collapses to a point.
\end{defin}

\begin{defin}[Nonevasiveness]
The properties of \textit{evasiveness} and \textit{nonevasiveness} of finite simplicial complexes can be defined recursively on the number of vertices in the following way. The empty simplicial complex~$\varnothing$ is, by definition, \textit{evasive}. A point, i.\,e., a $0$-simplex is, by definition, \textit{nonevasive}. Assume that $n\ge 2$ and the property of being evasive or nonevasive has already been defined for all simplicial complexes with less than $n$ vertices. Then a simplicial complex~$K$ with $n$ vertices is said to be \textit{nonevasive} if and only if there is a vertex $v$ of~$K$ such that both $\link(v,K)$ and~$\cost(v,K)$ are nonevasive. Otherwise, $K$ is said to be \textit{evasive}.
\end{defin}

We need the following two results from PL topology.

\begin{propos}[\cite{BrKu87}, Corollary~5]\label{propos_2collapse}
Let $K$ be a combinatorial manifold on the vertex set~$V$. Suppose that there exists a subset $U\subset V$ such that both complexes~$K_U$ and~$K_{V\setminus U}$ are collapsible. Then $K$ is a combinatorial sphere.
\end{propos}

\begin{cor}\label{cor_2collapse}
Suppose that a combinatorial manifold~$K$ contains a simplex~$\sigma$ such that $\cost\sigma$ is collapsible. Then $K$ is a combinatorial sphere.
\end{cor}

\begin{remark}
The proof of Proposition~\ref{propos_2collapse} relies on theory of regular neighborhoods. Recall that the \textit{first derived neighborhood} of a full subcomplex~$L$ of a combinatorial manifold~$K$ is the subcomplex~$N(L,K')$ of the first barycentric subdivision~$K'$ of~$K$ consisting of  all simplices  that meet~$L$ and all their faces. By~\cite[Theorem~5]{Zee63}, $N(L,K')$ is a combinatorial ball whenever $L$ is collapsible. So under the conditions of Proposition~\ref{propos_2collapse}, the complex~$K'$ is the union of the two combinatorial balls~$N(K_U,K')$ and~$N(K_{V\setminus U},K')$ along their common boundary. Therefore, $K'$ and hence~$K$ are combinatorial spheres.
\end{remark}

\begin{propos}[\cite{KSS84}, Proposition 1]\label{propos_nonev}
A nonevasive simplicial complex is collapsible.
\end{propos}

Proposition~\ref{propos_nonev} is important for us, since checking nonevasiveness is significantly faster than checking collapsibility. Indeed, the nonevasiveness can be checked directly by definition, and this procedure requires working with vertices and maximal simplices only without referring directly to simplices of intermediate dimensions.

Suppose that $K$ is a pure $d$-dimensional simplicial complex and we would like to check that $K$ is a combinatorial manifold. To a pair of simplices $\rho\subset\sigma$ of~$K$ such that $\dim\rho<\dim\sigma=d$, we assign the simplicial complex
$$
L_{\rho,\sigma} =\cost\bigl(\sigma\setminus\rho,\link(\rho,K)\bigr)=\{\eta\subseteq V\setminus\sigma \colon \rho\cup\eta\in K\}.
$$
The following proposition gives a sufficient condition for~$K$ to be a combinatorial manifold.

\begin{propos}\label{propos_combman_suff}
Suppose that, for each nonempty simplex $\rho\in K$ with $\dim\rho<d$, there exists a $d$-simplex $\sigma\in K$ such that $\sigma\supset\rho$ and $L_{\rho,\sigma}$ is collapsible. Then $K$ is a combinatorial manifold.
\end{propos}

\begin{proof}
Let us prove by reverse induction on~$k=d-1,\ldots,0$ that the link of every $k$-simplex $\rho\in K$ is a combinatorial $(d-k-1)$-sphere.

\textsl{Base of induction.} Suppose that $\dim\rho=d-1$. We know that there is a $d$-simplex $\sigma\supset\rho$ such that $L_{\rho,\sigma}$ is collapsible. Since $\dim L_{\rho,\sigma}=0$, it follows that $L_{\rho,\sigma}$ is a point. Hence, $\link(\rho,K)$ is a pair of points, i.\,e., a $0$-dimensional sphere.

\textsl{Inductive step.} Suppose that $k<d-1$ and assume that we have already proved that links of all simplices of $K$ of dimensions greater than $k$ are combinatorial spheres. Consider a $k$-simplex $\rho\in K$. For each vertex $v$ of $\link(\rho,K)$, the complex 
$$
\link\bigl( v, \link(\rho,K)\bigr)=\link(\{v\}\cup\rho,K)
$$
is a combinatorial $(d-k-2)$-sphere. Therefore, $\link(\rho,K)$ is a combinatorial $(d-k-1)$-manifold. We know that there is a $d$-simplex $\sigma\supset\rho$ such that the complex $$L_{\rho,\sigma}=\cost\bigl(\sigma\setminus\rho,\link(\rho,K)\bigr)$$ is collapsible. So, by Corollary~\ref{cor_2collapse}, $\link(\rho,K)$ is a  combinatorial  $(d-k-1)$-sphere.

For $k=0$, we obtain that the links of all vertices of~$K$ are combinatorial $(d-1)$-spheres. Thus, $K$ is a combinatorial  $d$-manifold.
\end{proof}

The following proposition is checked by a computer enumeration.

\begin{propos}\label{propos_check_nonev}
Let $K$ be one of the four simplicial complexes $K_1$, $K_2$, $K_3$, and~$K_4$ described in Tables~\ref{table_1234}--\ref{table_4}. Then, for each nonempty simplex $\rho\in K$ with $\dim\rho<16$, there exists a $16$-dimensional simplex $\sigma\in K$ such that $\sigma\supset\rho$ and $L_{\rho,\sigma}$ is nonevasive and hence collapsible.
\end{propos}

It follows from Propositions~\ref{propos_combman_suff} and~\ref{propos_check_nonev} that $K_1$, $K_2$, $K_3$, and~$K_4$ are $16$-dimensional combinatorial manifolds, which gives Proposition~\ref{propos_combman}.

\begin{remark}
Let us give a more detailed description of the algorithm whose implementation yields Proposition~\ref{propos_check_nonev}. Recall that each simplicial complex $K=K_i$ is  given by the list $\sigma_1,\ldots,\sigma_{286}$ of representatives of the $G_{351}$-orbits of its $16$-simplices. Since $L_{g(\rho),g(\sigma)}\cong L_{\rho,\sigma}$ for all $g\in G_{351}$, it is sufficient to examine the nonevasiveness of the complexes $L_{\rho,\sigma_s}$ only, where $\rho$ runs over nonempty proper faces of~$\sigma_s$. Moreover, once we have checked that a certain complex  $L_{\rho,\sigma_s}$ is nonevasive, we do not need any more to examine the nonevasiveness of  other complexes of the form $L_{\rho',\sigma_t}$ such that $\rho'$ lies in the $G_{351}$-orbit of~$\rho$. However, it turns out that generating complexes~$L_{\rho,\sigma_s}$ is a more time-consuming operation than checking that these complexes are nonevasive. Indeed, the simplices of~$L_{\rho,\sigma_s}$ are exactly the simplices $\tau\setminus\sigma_s$, where $\tau\in K$ and $\tau\supset\rho$. So to find all maximal simplices of~$L_{\rho,\sigma_s}$, we need to enumerate the $100386$ maximal simplices $\tau\in K$. If we tried to perform this enumeration separately for each pair $\rho\subset\sigma_s$, this would be rather time-consuming. So we proceed as follows.  For each $\sigma_s$, we perform the enumeration over all the $100386$ maximal simplices $\tau\in K$ only once, and for each such~$\tau$, we add the simplex $\tau\setminus \sigma_s$ to all simplicial complexes $L_{\rho,\sigma_s}$ with $\rho\subseteq \sigma_s\cap\tau$. Using this procedure, we generate all the simplicial complexes $L_{\rho,\sigma_s}$ with given~$s$ simultaneously. Then we start checking whether the generated complexes~$L_{\rho,\sigma_s}$ are nonevasive, and once a complex  $L_{\rho,\sigma_s}$ is checked to be nonevasive, we remove all complexes $L_{\rho',\sigma_t}$ with $\rho'$ in the $G_{351}$-orbit of~$\rho$ from the list of complexes to be examined. Being implemented as a C++ program under the same conditions as in Remark~\ref{remark_implement}, the running time of this algorithm is about $40$ minutes for each complex~$K_i$.
\end{remark}

\section{The normalizer of~$G_{351}$ and absence of additional symmetries}\label{section_no_add_sym}

In this section we prove Propositions~\ref{propos_normalizer} and~\ref{propos_SymK}.

Let $\Gamma$ be the oriented graph on the vertex set~$\F_{27}$ such that an oriented edge goes from~$a$ to~$b$ if and only if $b-a$ is a nonzero square. We denote the oriented edge from~$a$ to~$b$ by~$(a,b)$.  For any nonzero element $c\in\F_{27}$, exactly one of the two elements~$\pm c$ is a square; hence any two distinct elements $a,b\in\F_{27}$ are connected by an oriented edge in exactly one of the two directions. For a vertex~$a$ of~$\Gamma$, we denote by~$O_a$ (respectively, by~$I_a$) the set consisting of all vertices $b\ne a$ such that the graph~$\Gamma$ contains the outgoing edge~$(a,b)$ (respectively, the incoming  edge~$(b,a)$). Then $|O_a|=|I_a|=13$. We denote by~$\Sym(\Gamma)$ the symmetry group of the oriented graph~$\Gamma$.

\begin{propos}\label{propos_orgraph_symmetries}
We have $\Sym(\Gamma)=G_{351}\rtimes \rC_3$, where the factor~$\rC_3$ is generated by the Frobenius automorphism $F\colon x\mapsto x^3$.
\end{propos} 

The proof of this proposition will be decomposed into several lemmas.

\begin{lem}\label{lem_orgraph_symmetries_inclusion}
We have, $\Sym(\Gamma)\supseteq G_{351}\rtimes \rC_3$. Besides, the subgroup~$G_{351}\subset \Sym(\Gamma)$ acts transitively on the edges of~$\Gamma$.
\end{lem}

\begin{proof}
It follows immediately from the definition of~$\Gamma$ that it is symmetric with respect to additive shifts $x\mapsto x+c$, where $c\in \F_{27}$, multiplicative shifts by squares $x\mapsto cx$, where $c\in(\F_{27}^{\times})^2$, and all automorphisms of the field~$\F_{27}$.
Besides, if $(a,b)$ is an edge of~$\Gamma$, then $b-a$ is a square and hence the transformation $x\mapsto (b-a)x+a$ belongs to~$G_{351}$ and takes the edge~$(0,1)$ to~$(a,b)$. 
\end{proof}

\begin{lem}\label{lem_sym_stab_mnogo}
Suppose that an  element $g\in\Sym(\Gamma)$ stabilizes every of the vertices~$0$, $1$, $\alpha$, $\alpha^3$, and~$\alpha^9$. Then $g=\mathrm{id}$.
\end{lem}

\begin{proof}
Whenever $g$ stabilizes a vertex~$x$, it necessarily stabilizes each of the sets~$O_x$ and~$I_x$. So $g$ stabilizes every of the sets~$O_0$, $O_1$, $O_{\alpha}$, $O_{\alpha^3}$, and~$O_{\alpha^9}$, and the corresponding sets~$I_x$. Since $O_0=(\F_{27}^{\times})^2$, we obtain that $g$ takes squares to squares and non-squares to non-squares.

First, let us prove that $g$ stabilizes every square~$\alpha^k$. To do this, for every square~$\alpha^k$ different from~$1$, $\alpha$, $\alpha^3$, and~$\alpha^9$, find out which of the sets $O_1$, $O_{\alpha}$, $O_{\alpha^3}$, and~$O_{\alpha^9}$ contain~$\alpha^k$, that is, which of the elements $\alpha^k-1$, $\alpha^k-\alpha$, $\alpha^k-\alpha^3$, and $\alpha^k-\alpha^9$ are squares. The result of an easy calculation is shown in Table~\ref{table_alpha_member}. In this table $+$ (respectively, $-$) indicates that the element belongs (respectively, does not belong) to the set. We see that the columns in the table are pairwise different, so the squares $\alpha^2$, $\alpha^4$, $\alpha^5$, $\alpha^6$, $\alpha^7$, $\alpha^8$, $\alpha^{10}$, $\alpha^{11}$, and~$\alpha^{12}$ are pairwise distinguished by membership in the sets $O_1$, $O_{\alpha}$, $O_{\alpha^3}$, and~$O_{\alpha^9}$. Thus,  $g$ stabilizes every square~$\alpha^k$.

\begin{table}
\caption{Membership of elements in sets}\label{table_alpha_member}
\begin{tabular}{|l|c|c|c|c|c|c|c|c|c|}
\hline
$\phantom{\strut^2}$ & $\alpha^2$ & $\alpha^4$ & $\alpha^5$ & $\alpha^6$ & $\alpha^7$ & $\alpha^8$ & $\alpha^{10}$ & $\alpha^{11}$ & $\alpha^{12}$\\
\hline
$O_1$ & $+$ & $-$ & $+$ & $+$ & $-$ & $-$ & $-$ & $-$ & $-$\\
\hline
$O_{\alpha}$ & $+$ & $+$ & $-$ & $+$ & $+$ & $-$ & $+$ & $-$ & $-$\\
\hline
$O_{\alpha^3}$ & $-$ & $+$ & $+$ & $+$ & $-$ & $+$ & $-$ & $-$ & $+$\\
\hline
$O_{\alpha^9}$ & $+$ & $-$ & $+$ & $-$ & $-$ & $-$ & $+$ & $+$ & $+$\\
\hline
\end{tabular}
\end{table}

Now, consider all non-squares $-\alpha^k$, where $k=0,\ldots,12$. The set $O_{-1}\cap O_0$ consists of all elements $x\in\F_{27}$ such that both $x$ and~$x+1$ are nonzero squares. It is easy to see that $O_{-1}\cap O_0=\{\alpha,\alpha^3,\alpha^4,\alpha^9,\alpha^{10},\alpha^{12}\}$. The sets $O_{-\alpha^k}\cap O_0$ are obtained from $O_{-1}\cap O_0$ by multiplication by~$\alpha^k$ and hence are pairwise different. On the other hand, each set~$O_{-\alpha^k}\cap O_0$ consists of squares and therefore is stabilized by~$g$. So for all non-squares~$x$, we have
$$
O_x\cap O_0=g( O_x\cap O_0 ) = O_{g(x)}\cap O_0
$$
and hence $g(x)=x$.
\end{proof}

\begin{lem}\label{lem_sym_stab_malo}
An element $g\in\Sym(\Gamma)$ satisfies~$g(0)=0$ and~$g(1)=1$ if and only if $g$ is a power of the Frobenius automorphism~$F$.
\end{lem}

\begin{proof}
Obviously, $F(0)=0$ and $F(1)=1$. Now, we assume that $g(0)=0$ and $g(1)=1$ and prove that $g$ is a power of~$F$. Since $g(0)=0$ and $g(1)=1$, we see that $g$ stabilizes the set~$O_0\cap O_1$. This set consists of all elements $x\in\F_{27}$ such that both $x$ and~$x-1$ are nonzero squares. We have $O_0\cap O_1=\{\alpha,\alpha^2,\alpha^3,\alpha^5,\alpha^6,\alpha^9\}$. If $g(x)=y$, then $g$ maps the set $O_0\cap O_1\cap O_x$ bijectively onto the set $O_0\cap O_1\cap O_y$. However, it is easy to see that each of the sets $O_0\cap O_1\cap O_{\alpha^k}$ with $k=1,3,9$ consists of three elements, while each of the sets $O_0\cap O_1\cap O_{\alpha^k}$ with $k=2,5,6$ consists of two elements. Indeed, $O_0\cap O_1\cap O_{\alpha}=\{\alpha^2,\alpha^3,\alpha^6\}$, $O_0\cap O_1\cap O_{\alpha^2}=\{\alpha^3,\alpha^5\}$, and the other four sets are obtained from those two by the action of the Frobenius automorphism, which  permutes cyclically $\alpha$, $\alpha^3$, and $\alpha^9$ and permutes cyclically $\alpha^2$, $\alpha^6$, and $\alpha^5$. Thus, $g$ stabilizes each of the two sets $\{\alpha,\alpha^3,\alpha^9\}$ and $\{\alpha^2,\alpha^5,\alpha^6\}$. Hence there is a number~$m$ such that $F^mg(\alpha)=\alpha$. Then $F^mg(\alpha^3)$ is either $\alpha^3$ or~$\alpha^9$. Since the set $O_0\cap O_1\cap O_{\alpha}=\{\alpha^2,\alpha^3,\alpha^6\}$ is stabilized by~$g$, it follows that $F^mg(\alpha^3)=\alpha^3$ and hence $F^mg(\alpha^9)=\alpha^9$. By Lemma~\ref{lem_sym_stab_mnogo}, $F^mg=\mathrm{id}$. Therefore, $g=F^{-m}$.
\end{proof}

\begin{lem}\label{lem_size_Sym}
$|\Sym(\Gamma)|= 3\cdot 351=1053$.
\end{lem}

\begin{proof}
The equality follows, since $\Sym(\Gamma)$ acts transitively on the set of $\binom{27}{2}=351$ oriented edges of~$\Gamma$ and, by Lemma~\ref{lem_sym_stab_malo}, the stabilizer of the oriented edge $(0,1)$ has order~$3$.
\end{proof}

Proposition~\ref{propos_orgraph_symmetries} follows from Lemmas~\ref{lem_orgraph_symmetries_inclusion} and~\ref{lem_size_Sym}.

\begin{proof}[Proof of Proposition~\ref{propos_normalizer}]
Let $\Sym^{\pm}(\Gamma)$ be the group consisting of all permutations of~$\F_{27}$  that either preserve the directions of all edges of~$\Gamma$ or reverse the directions of all edges of~$\Gamma$ simultaneously. Since $-1$ is not a square in~$\F_{27}$,  the sign reversal $S\colon x\mapsto -x$ reverses the directions of all edges of~$\Gamma$, so $\Sym(\Gamma)$ is an index~$2$ subgroup of $\Sym^{\pm}(\Gamma)$ and $\Sym^{\pm}(\Gamma)$ is generated by $\Sym(\Gamma)$ and~$S$. It follows from Proposition~\ref{propos_orgraph_symmetries} that $\Sym^{\pm}(\Gamma)$ is exactly the $2106$-element group $G_{351}\rtimes (\rC_2\times \rC_3)$ from Proposition~\ref{propos_normalizer}. So we need to prove that $N(G_{351})=\Sym^{\pm}(\Gamma)$. We have $\Sym^{\pm}(\Gamma)\subseteq N(G_{351})$, since the permutations~$F$ and~$S$ normalize the subgroup~$G_{351}$. Hence we need only to prove the inverse inclusion.

Suppose that a permutation $h$ of the set~$\F_{27}$ does not belong to~$\Sym^{\pm}(\Gamma)$ and prove that then it does not belong to~$N(G_{351})$. Since $h\notin\Sym^{\pm}(\Gamma)$, there exist edges~$e_1$ and~$e_2$ of~$\Gamma$ such that $h(e_1)$ is an edge of~$\Gamma$ and~$h(e_2)$ is not an edge of~$\Gamma$. By Lemma~\ref{lem_orgraph_symmetries_inclusion}, the group~$G_{351}$ acts transitively on edges of~$\Gamma$; hence, there is an element $g\in G_{351}$ satisfying $g(e_1)=e_2$. Then the permutation $hgh^{-1}$ takes the edge~$h(e_1)$  to the non-edge~$h(e_2)$. Therefore, $hgh^{-1}\notin \Sym(\Gamma)$ and hence $hgh^{-1}\notin G_{351}$. Thus, $h\notin N(G_{351})$.
\end{proof}

To prove Proposition~~\ref{propos_SymK}, we need the following lemma.

\begin{lem}\label{lem_not_transitive}
Suppose that $i\in\{1,2,3,4\}$ and $e$ is an edge of~$K_i$. Then the group~$\Sym(K_i)$ does not contain an element that swaps the endpoints of~$e$.
\end{lem}

\begin{proof}
By the construction of~$K_i$, the group~$\Sym(K_i)$ contains~$G_{351}$ and hence acts transitively on the set of undirected edges of~$K_i$. So it suffices for us to prove the lemma only for some one edge of~$K_i$, say for the edge $\{1,2\}$ in the numbering introduced in Section~\ref{section_result}. To do this, we improve our method from Step~3 of the scheme of proof of Theorem~\ref{theorem_main_detail}, see Section~\ref{section_scheme}. Namely, for each $15$-simplex $\tau\in K_i$ that contains both~$1$ and~$2$, we can consider the two $14$-dimensional faces of it, $\rho_1=\tau\setminus\{1\}$ and~$\rho_2=\tau\setminus\{2\}$. For each of them, we can compute the number~$s(\rho_k)$ of $16$-simplices $\sigma\in K_i$  such that $\sigma\supset\rho_k$. Thus, the ordered pair of numbers $s(\tau)=\bigl(s(\rho_1),s(\rho_2)\bigr)$ corresponds to~$\tau$. Now, for each ordered pair of positive integers~$(p,q)$, we can compute the number~$N_{pq}$ of $15$-simplices~$\tau$ such that $\tau\supset\{1,2\}$ and $s(\tau)=(p,q)$. The matrices $(N_{pq})$ for the triangulations~$K_1$, $K_2$, $K_3$, and~$K_4$ are given in Table~\ref{table_N_matrix}. If there existed a symmetry of~$K_i$ that swaps~$1$ and~$2$, then the corresponding matrix~$(N_{pq})$ would be symmetric. The lemma follows, since none of the matrices in Table~\ref{table_N_matrix} is symmetric.
\end{proof}

\begin{table}[t]
\caption{The matrices $(N_{pq})$ for~$K_1$, $K_2$, $K_3$, and~$K_4$}\label{table_N_matrix}
\begin{tabular}{l|c|ccccccc|}
\cline{2-9}
& & 3 & 4 & 5 & 6 & 7 & 8 & 9 \\
\cline{2-9}
&3 & 10860 & 22509 & 13809 & 5375 & 1374 & 360 & 16\\
&4 & 22504 & 62261 & 31272 & 10187 & 2247 & 624 & 21\\
&5 & 14207 & 31196 & 20867 & 6602 & 1481 & 372 & 13\\
$\boldsymbol{K_1}:\quad$&6 & 5267 & 10221 & 6553 & 2905 & 737 & 183 & 13\\
&7 & 1400 & 2166 & 1488 & 656 & 242 & 59 & 6\\
&8 & 388 & 567 & 359 & 171 & 76 & 40 & 1\\
&9 & 16 & 24 & 14 & 5 & 6 & 0 & 0\\
\cline{2-9}
\multicolumn{9}{c}{}\\
\cline{2-9}
& & 3 & 4 & 5 & 6 & 7 & 8 & 9 \\
\cline{2-9}
&3&14532 &  21847 &  15078  & 6996 &  2090  & 322 &  89\\
&4 &21975 &  52378  &  28416  & 10895  & 2745  & 389  & 79\\
&5 &15408 &  28271 &  20266 &  7554  & 1862 &  268 &  63  \\
$\boldsymbol{K_2}:\quad$&6 &6956 &  10614 &  7825 &  3847  & 1088  & 138  & 48   \\
&7 &2003 &  2796  & 1886 &  1035 &  400  & 69  & 20   \\
&8 &320 &  391  & 237  & 146 &  73  & 19   & 5   \\
&9 & 72  & 106 &  50  & 31  & 18  &  4 &  0 \\
\cline{2-9}
\multicolumn{9}{c}{}\\
\cline{2-9}
&& 3 & 4 & 5 & 6 & 7 & 8 & 9 \\
\cline{2-9}
&3 & 14304 & 21634 & 15634 & 6324 & 1901 & 368 & 28\\
&4 & 21663 & 51651 & 30266 & 10093 & 2596 & 408 & 28\\
&5 & 15911 & 30125 & 21767 & 7505 & 1886 & 300 & 15\\
$\boldsymbol{K_3}:\quad$&6 & 6320 & 9940 & 7467 & 3358 & 928 & 163 & 13\\
&7 & 1806 & 2673 & 1904 & 936 & 328 & 71 & 7\\
&8 & 353 & 397 & 313 & 168 & 60 & 19 & 0\\
&9 & 20 & 35 & 15 & 7 & 11 & 1 & 0\\
\cline{2-9}
\multicolumn{9}{c}{}\\
\cline{2-9}
&& 3 & 4 & 5 & 6 & 7 & 8 & 9 \\
\cline{2-9}
&3 & 12288 & 21553 & 15249 & 4858 & 1361 & 276 & 75\\
&4 & 21376 & 56670 & 34129 & 8632 & 2339 & 406 & 99\\
&5 & 15438 & 34057 & 23961 & 6323 & 1526 & 307 & 71\\
$\boldsymbol{K_4}:\quad$&6 & 4953 & 8724 & 6454 & 2408 & 586 & 84 & 36\\
&7 & 1259 & 2286 & 1586 & 625 & 247 & 53 & 8\\
&8 & 251 & 409 & 284 & 112 & 48 & 17 & 9\\
&9 & 75 & 90 & 79 & 27 & 9 & 7 & 0\\
\cline{2-9}
\end{tabular}
\end{table}

\begin{cor}\label{cor_sym_cont}
For every $i=1,2,3,4$, the group~$\Sym(K_i)$ is contained in~$\Sym(\Gamma)$.
\end{cor}

\begin{proof}
We know that~$\Sym(K_i)$ contains~$G_{351}$ and hence acts transitively on undirected edges of~$\Gamma$. So if~$\Sym(K_i)$ contained an element not belonging to~$\Sym(\Gamma)$, then $\Sym(K_i)$ would act transitively on directed edges of~$K_i$, which is not true by Lemma~\ref{lem_not_transitive}.
\end{proof}

\begin{proof}[Proof of Proposition~\ref{propos_SymK}]
By Proposition~\ref{propos_orgraph_symmetries} and Corollary~\ref{cor_sym_cont}, we have 
$$
G_{351}\subseteq \Sym(K_i)\subseteq \Sym(\Gamma)=G_{351}\rtimes \rC_3,
$$
where $\rC_3$ is generated by the Frobenius automorphism~$F$. Hence $\Sym(K_i)$ is either $G_{351}$ or $G_{351}\rtimes \rC_3$. However, as we have already mentioned in Section~\ref{section_result}, a direct computation shows that the group~$\rC_3$ generated by~$F$ acts freely on the set of the $24$ simplicial complexes in Theorem~\ref{theorem_main_detail}. This means that $F\notin\Sym(K_i)$. Thus, $\Sym(K_i)=G_{351}$.
\end{proof}

\section{Triangulations with smaller symmetry groups}
\label{section_moves}

Let $K$ be a $(3d/2+3)$-vertex triangulation of a $d$-manifold like a projective plane. Each of the previously known examples $K=\RP^2_6$, $\CP^2_9$, $\HP^2_{15}$, $\tHP^2_{15}$, and~$\ttHP^2_{15}$ satisfies the following condition, see~\cite[Section~3]{BrKu92}:

\begin{itemize}
\item[$(*)$] $K$ contains three $(d/2)$-simplices~$\Delta_1$, $\Delta_2$, and~$\Delta_3$ such that 
\begin{equation}\label{eq_triple}
\link(\Delta_1,K)=\partial\Delta_2,\qquad
\link(\Delta_2,K)=\partial\Delta_3,\qquad
\link(\Delta_3,K)=\partial\Delta_1.
\end{equation}
\end{itemize}

The author is indebted to Denis Gorodkov for pointing out the importance of this condition.

We shall conveniently use the following terminology. A triple~$(\Delta_1,\Delta_2,\Delta_3)$ of $(d/2)$-simplices satisfying~\eqref{eq_triple} will be called a \textit{distinguished triple}. If  $(\Delta_1,\Delta_2,\Delta_3)$ is a distinguished triple, then $K$ contains the subcomplex 
\begin{equation}\label{eq_distinguished_subcomp}
J=(\Delta_1*\partial\Delta_2)\cup(\Delta_2*\partial\Delta_3)\cup(\Delta_3*\partial\Delta_1),\end{equation}
where $*$ denotes the join of simplicial complexes.
This subcomplex will be referred to as a \textit{distinguished subcomplex}.

Brehm and K\"uhnel made the following observation.

\begin{propos}[\cite{BrKu92}, Lemma 1]\label{propos_move}
Suppose that a combinatorial $d$-manifold~$K$ contains a distinguished subcomplex of the form~\eqref{eq_distinguished_subcomp}. Then a replacement of this subcomplex by
\begin{equation}\label{eq_distinguished_replace}
\wJ=(\partial\Delta_1*\Delta_2)\cup(\partial\Delta_2*\Delta_3)\cup(\partial\Delta_3*\Delta_1)
\end{equation}
yields a combinatorial manifold~$\widetilde{K}$ that is PL homeomorphic to~$K$.
\end{propos}

\begin{remark}
A replacement of a subcomplex of the form $\Delta_1*\partial\Delta_2$ by~$\partial\Delta_1*\Delta_2$ is called a \textit{bistellar move} (or \textit{flip}). Since both $\Delta_1*\partial\Delta_2$ and~$\partial\Delta_1*\Delta_2$ are combinatorial balls with the same boundary $\partial\Delta_1*\partial\Delta_2$, this operation yields a PL manifold homeomorphic to the initial one. Nevertheless, there is an important difference between a single bistellar move and the move in Proposition~\ref{propos_move} consisting of three simultaneous bistellar moves. Namely, a single bistellar move $\Delta_1*\partial\Delta_2\rightsquigarrow\partial\Delta_1*\Delta_2$ is well defined only when $\Delta_1*\partial\Delta_2$ is a full subcomplex of the initial manifold~$K$, i.\,e., $\Delta_2\notin K$. On the contrary, the operation in Proposition~\ref{propos_move} is always defined. To show this, one only needs to check that any simplex~$\sigma$ that belongs to~$K$ and to the new subcomplex~$\wJ$ belongs to the old subcomplex~$J$ as well. Indeed, up to a cyclic permutation of~$\Delta_1$, $\Delta_2$, and~$\Delta_3$, any simplex~$\sigma\in\wJ$ has the form $\sigma=\sigma_1*\sigma_2$, where $\sigma_1\subsetneq\Delta_1$ and $\sigma_2\subseteq\Delta_2$. If $\sigma_1=\varnothing$ or $\sigma_2\ne \Delta_2$, then $\sigma\in J$. If $\sigma_1\ne\varnothing$ and $\sigma_2=\Delta_2$, then $\sigma\notin K$, since $\sigma_1\notin\link (\Delta_2,K)=\partial\Delta_3$. 
\end{remark}

The replacement of a subcomplex~$J$ of the form~\eqref{eq_distinguished_subcomp} by the subcomplex~$\wJ$ of the form~\eqref{eq_distinguished_replace} will be called a \textit{triple flip} corresponding to the distinguished triple~$(\Delta_1,\Delta_2,\Delta_3)$. The fact that the triple flip corresponding to a distinguished triple is always defined immediately implies the following important assertion.

\begin{propos}\label{propos_triple_commute}
 Suppose that $J_1$ and~$J_2$ are two distinguished subcomplexes of~$K$ with disjoint interiors, i.\,e., without common $d$-simplices. Then $J_2$ remains a distinguished subcomplex (and so the corresponding triple flip remains defined) after performing the triple flip corresponding to~$J_1$, and vice versa. Moreover, the triple flips corresponding to~$J_1$ and~$J_2$ commute with each other.
\end{propos}

\begin{cor}\label{cor_triple_commute}
 Suppose that $J_1,\ldots,J_k$ are distinguished subcomplexes of~$K$ with pairwise disjoint interiors. Then we can perform the triple flips corresponding to $J_1,\ldots,J_k$ in any order, and the resulting combinatorial manifold is  independent of this order.
\end{cor}

Before studying the octonionic case, we recall how the distinguished subcomplexes look like in the real, complex, and quaternionic cases.

\smallskip

\textsl{Case 1. Suppose that $d=2$ and~$K=\RP^2_6$.} Then any one-dimensional simplex of~$\RP^2_6$ enters a distinguished triple and all the $15$ one-dimensional simplices of~$\RP^2_6$ are decomposed into $5$ distinguished triples. The triple flip corresponding to any of these distinguished triples transforms~$\RP^2_6$ to a combinatorial manifold that is isomorphic to~$\RP^2_6$ again.

\smallskip

\textsl{Case 2. Suppose that $d=4$ and~$K=\CP^2_9$.} Then $\CP^2_9$ contains a unique distinguished triple, cf.~\cite[\S 3]{MoYo91}. Recall that the $54$-element group $\Sym(\CP^2_9)$ acts on the $36$ four-dimensional simplices with two orbits consisting of $27$ and~$9$ simplices, respectively. The $9$ four-dimensional simplices in the second orbit and their faces form the distinguished subcomplex~$J$. So $J$ is invariant under the action of  $\Sym(\CP^2_9)$. The corresponding triple flip  transforms~$\CP^2_9$ to a combinatorial manifold that is isomorphic to~$\CP^2_9$ again.

\smallskip

\textsl{Case 3. Suppose that $d=8$ and~$K=\HP^2_{15}$ is the most symmetric of the three Brehm--K\"uhnel triangulations.} Then $\HP^2_{15}$ contains exactly five distinguished triples and the group $\Sym(\HP^2_{15})\cong \rA_5$ acts transitively on them, see~\cite[Section~3]{BrKu92}. The corresponding five distinguished subcomplexes~$J_1,\ldots,J_5$ have pairwise disjoint interiors. So we may perform the corresponding five triple flips independently of each other. If one performs simultaneously all the five triple flips corresponding to~$J_1,\ldots,J_5$, then the obtained combinatorial manifold will be isomorphic to~$\HP^2_{15}$ again. Nevertheless, performing only some of the five triple flips, one can obtain new triangulations. This is exactly the way how Brehm and K\"uhnel constructed the two other triangulations. Namely, performing any one (or any four) of the five triple flips leads to a combinatorial manifold isomorphic to~$\tHP^2_{15}$, and performing any two (or any three) of the five triple flips leads to a combinatorial manifold isomorphic to~$\ttHP^2_{15}$. For each $k$, the result is independent of which $k$ of the $5$ flips are performed, since the group $\Sym(\HP^2_{15})\cong \rA_5$ acts transitively on the $k$-element subsets of $\{1,2,3,4,5\}$. Note that the symmetry groups $\Sym\left(\tHP^2_{15}\right)\cong \rA_4$ and $\Sym\left(\ttHP^2_{15}\right)\cong \rS_3$ are exactly the subgroups of~$\rA_5$ that stabilize a distinguished subcomplex and the union of two distinguished subcomplexes, respectively.

\smallskip

Now, we are ready to consider the case $d=16$, that is, the four combinatorial manifolds $K_1$, $K_2$, $K_3$, and~$K_4$ from Theorem~\ref{theorem_main_detail}.

\begin{theorem}\label{theorem_smaller_groups}
\textnormal{(1)} The combinatorial manifolds~$K_1$ and~$K_4$ contain no distinguished subcomplexes.

\textnormal{(2)} The combinatorial manifold~$K_2$ contains exactly $351$ distinguished subcomplexes. One of them is the subcomplex
$$
J=(\Delta_1*\partial\Delta_2)\cup(\Delta_2*\partial\Delta_3)\cup(\Delta_3*\partial\Delta_1), 
$$
where 
\begin{align*}
\Delta_1&=\{1,2,5,6,9,10,11,16,25\},\\
\Delta_2&=\{3,4,7,8,12,13,14,15,17 \},\\
\Delta_3&=\{18,19,20,21,22,23,24,26,27\}.
\end{align*} 
The $351$ subcomplexes $g(J)$, where $g\in \Sym(K_2)=G_{351}$, have pairwise disjoint interiors and are exactly all the distinguished subcomplexes of~$K_2$.

\textnormal{(3)} For a subset $S\subseteq G_{351}$, let $K_S$ be the simplicial complex obtained from~$K_2$ by replacing all subcomplexes $g(J)$ such that $g\in S$ by the corresponding complexes $g\bigl(\widetilde{J}\,\bigr)$, where  
$$
\widetilde{J}=(\partial\Delta_1*\Delta_2)\cup(\partial\Delta_2*\Delta_3)\cup(\partial\Delta_3*\Delta_1).
$$
Then $K_S$ is a $27$-vertex combinatorial $16$-manifold like the octonionic projective plane and is PL homeomorphic to~$K_2=K_{\varnothing}$. The symmetry group~$\Sym(K_S)$ is exactly the subgroup of~$G_{351}$ consisting of all $g$ satisfying $gS=S$. Two combinatorial manifolds~$K_{S_1}$ and $K_{S_2}$ are isomorphic if and only if there exists $g\in G_{351}$ such that $gS_1=S_2$. Besides, $g(J)$ with $g\in G_{351}\setminus S$ and $g\bigl(\widetilde{J}\,\bigr)$ with $g\in S$ are exactly all distinguished subcomplexes of~$K_S$.

\textnormal{(4)} $K_{G_{351}}=K_3$.

\textnormal{(5)} The number of combinatorially distinct combinatorial manifolds of the form~$K_S$ is exactly   
$$
\frac1{351}\cdot\left(2^{351}+13\cdot 2^{118}+81\cdot 2^{29}\right).
$$
So together with the two exceptional combinatorial manifolds~$K_1$ and~$K_4$, we obtain 
$$
\frac1{351}\cdot\left(2^{351}+13\cdot 2^{118}+81\cdot 2^{29}\right)+2
$$
$27$-vertex combinatorial $16$-manifolds like the octonionic projective plane. Moreover, the distribution of these combinatorial manifolds with respect to symmetry groups is exactly as in Table~\ref{table_num_triang}.
\end{theorem}

We will prove this theorem in the next section. Now, we would like to stress that four new phenomena occur in dimension~$16$:

\begin{itemize}
\item $K_1$ and~$K_4$ are the first examples of $(3d/2+3)$-vertex combinatorial $d$-manifolds like a projective plane without distinguished subcomplexes.
\item If we perform in~$K_2$ all the possible triple flips simultaneously, we obtain a combinatorial manifold~$K_3$, which is not isomorphic to~$K_2$. In all previously known examples in dimensions~$2$, $4$, and~$8$ performing all the possible triple flips simultaneously led to a combinatorial manifold isomorphic to the initial one.
\item The simplices $\Delta_1,\Delta_2,\Delta_3\in K_2$ that constitute the distinguished triple from Claim~(2) of  Theorem~\ref{theorem_smaller_groups} lie in pairwise different $\Sym(K_2)$-orbits. (This follows immediately from the fact that the subcomplexes~$g(J)$ with $g\in \Sym(K_2)$ are pairwise different.) On the contrary, in all previously known examples of distinguished triples~$(\Delta_1,\Delta_2,\Delta_3)$ in dimensions~$2$, $4$, and~$8$ there existed an element of the symmetry group permuting $\Delta_1$, $\Delta_2$, and~$\Delta_3$ cyclically. 
\item Also a new phenomenon is the existence of (a huge amount of) $27$-vertex combinatorial $16$-manifolds like the octonionic projective plane with trivial symmetry groups. All known examples of $(3d/2+3)$-vertex combinatorial $d$-manifolds like a projective plane in dimensions $d=2,4,8$ have nontrivial symmetry groups. (After writing the present paper, the author~\cite{Gai23b} constructed a lot of $15$-vertex triangulations of $\HP^2$ with trivial symmetry group.)
\end{itemize}

\section{Proof of Theorem~\ref{theorem_smaller_groups}}\label{section_proof_smaller_groups}

Suppose that $X$ is a pure $d$-dimensional simplicial complex on a finite vertex set~$V$. Let $\sigma$ be a $d$-simplex of~$X$ and let $v$ be a vertex of~$X$ not belonging to~$\sigma$. We denote by $\nu_X(\sigma,v)$ the number of $d$-simplices $\tau\in X$ such that $v\in\tau$ and $\dim(\sigma\cap\tau)=d-1$. The numbers $\nu_X(\sigma,v)$ for various $v\notin\sigma$ will be called \textit{$\nu$-parameters} for~$\sigma$. 

If $X$ is a weak $d$-pseudomanifold, then it contains exactly $d+1$ simplices $\tau$ such that $\dim\tau=d$ and $\dim(\sigma\cap\tau)=d-1$. Hence,

\begin{equation*}
\sum_{v\in V\setminus\sigma} \nu_X(\sigma,v)=d+1.
\end{equation*}

\begin{propos}\label{propos_nu_param}
Suppose that $K$ is a  combinatorial $d$-manifold like a projective plane with $3d/2+3$ vertices. Then 
\begin{enumerate}
\item $\nu_K(\sigma,v)\le d/2$ for all~$\sigma$ and~$v$, 
\item a $d$-simplex $\sigma\in K$ belongs to a distinguished subcomplex if and only if there is a vertex $v\notin\sigma$ satisfying $\nu_K(\sigma,v)=d/2$,
\item moreover, a $d$-simplex $\sigma\in K$ belongs to exactly $q$ different distinguished subcomplexes, where $q$ is the number of vertices $v\notin\sigma$ satisfying $\nu_K(\sigma,v)=d/2$.
\end{enumerate} 
\end{propos}

\begin{proof}
Recall that $K$ satisfies the neighborliness condition~(b) and the complementarity condition~(c) from Section~\ref{section_intro}.

Let $\sigma$ be a $d$-simplex of~$K$. Suppose that $\nu_K(\sigma,v)=k$. Then there are exactly $k$ vertices $u\in \sigma$ satisfying $$(\sigma\setminus\{u\})\cup\{v\}\in K.$$ We denote by~$\rho$ the set consisting of all such vertices~$u$, and put 
$$\Delta_1=\sigma\setminus\rho, 
\qquad 
\Delta_2=\rho\cup\{v\}, 
\qquad 
\Delta_3=V\setminus(\sigma\cup\{v\}).
$$
Then any set obtained from $\Delta_1\cup\Delta_2$ by removing a vertex from~$\Delta_2$ belongs to~$K$, so $K$ contains the subcomplex $\Delta_1*\partial\Delta_2$. By the complementarity condition, we obtain that $K$ contains no simplex of the form $\Delta_3\cup\{u\}$ with $u\in \Delta_2$. On the other hand, the set~$\Delta_3$ consists of $d/2+1$ elements and hence, by the neighborliness condition, is a $(d/2)$-simplex of~$K$. It follows that $\link(\Delta_3,K)$ is contained in~$\Delta_1$, so it has at most $d+1-k$ vertices. However, since $K$ is a  combinatorial $d$-manifold, we obtain that $\link(\Delta_3,K)$ is a  combinatorial $(d/2-1)$-sphere and hence has at least $d/2+1$ vertices. Therefore $k\le d/2$.

Moreover, if $k=d/2$, then $\link(\Delta_3,K)=\partial\Delta_1$. So $K$ contains the subcomplex $\Delta_3*\partial\Delta_1$. Repeating the same argument, we similarly obtain that $K$ contains the subcomplex $\Delta_2*\partial\Delta_3$, too. Then $K$ contains the distinguished subcomplex 
$$
(\Delta_1*\partial\Delta_2)\cup(\Delta_2*\partial\Delta_3)\cup(\Delta_3*\partial\Delta_1).
$$

Thus, to each vertex $v\notin \sigma$ with $\nu_K(\sigma,v)=d/2$, we have assigned a distinguished subcomplex containing~$\sigma$. Vice versa, suppose that  
$$
J=(\Delta_1*\partial\Delta_2)\cup(\Delta_2*\partial\Delta_3)\cup(\Delta_3*\partial\Delta_1)
$$
is a distinguished subcomplex containing~$\sigma$. Permuting cyclically~$\Delta_1$, $\Delta_2$, and~$\Delta_3$, we may achieve that $\sigma$ is contained in~$\Delta_1*\partial\Delta_2$. We assign to~$J$ a unique vertex~$v\in\Delta_2$ that does not belong to~$\sigma$. Then $\nu_K(\sigma,v)=d/2$. 

It is easy to see that the constructed correspondence between vertices $v\notin\sigma$ satisfying $\nu_K(\sigma,v)=d/2$ and distinguished subcomplexes containing~$\sigma$ is bijective, which completes the proof of the proposition.
\end{proof}

For each of the four combinatorial $16$-manifolds~$K_1$, $K_2$, $K_3$, and~$K_4$, all $\nu$-parameters can be easily calculated using a computer. The result of this calculation is as follows.

\begin{propos}\label{propos_nu_param_calc}
\textnormal{(1)} All $\nu$-parameters $\nu_{K_1}(\sigma,v)$ and $\nu_{K_4}(\sigma,v)$ do not exceed~$7$.

\textnormal{(2)} For each $16$-simplex $\sigma\in K_i$, where $i$ is either~$2$ or~$3$, at most one of the corresponding $\nu$-parameters $\nu_{K_i}(\sigma,v)$ can be equal to~$8$.

\textnormal{(3)} There are exactly $27$ \ $G_{351}$-orbits of $16$-simplices $\sigma \in K_2$ for which one of the $\nu$-parameters $\nu_{K_2}(\sigma,v)$ is equal to~$8$. These are exactly all $16$-simplices $\sigma \in K_2$ that do not belong to~$K_3$. Moreover, these $27\cdot 351=9477$ simplices~$\sigma$ are exactly all $16$-simplices that belong to the union $\bigcup_{g\in G_{351}}g(J)$, where $J$ is the complex from Claim~(2) of Theorem~\ref{theorem_smaller_groups}.

\textnormal{(4)} The same remains true if we swap~$K_2$ and~$K_3$, and replace~$J$ with the complex~$\widetilde{J}$ from Claim~(3) of Theorem~\ref{theorem_smaller_groups}.
\end{propos}

Claims~(1), (2), and~(4) of Theorem~\ref{theorem_smaller_groups} immediately follow from Propositions~\ref{propos_nu_param} and~\ref{propos_nu_param_calc}. Moreover, Claim~(1) of Proposition~\ref{propos_nu_param_calc} in fact implies the following assertion.

\begin{cor}\label{cor_no8-8}
 Neither of the combinatorial manifolds~$K_1$ and~$K_4$ contains an $8$-simplex whose link is the boundary of an $8$-simplex.
\end{cor}

Since $K_2$ and~$K_3$ are subcomplexes of the same $26$-simplex with the vertex set~$[27]$, we can consider the simplicial complexes $K_2\cup K_3$ and $K_2\cap K_3$. The former of them is pure but the latter is not, so instead of the intersection~$K_2\cap K_3$ we conveniently consider the following smaller complex. Namely, let $K_{2,3}$ be the simplicial complex consisting of all $16$-simplices belonging to both $K_2$ and~$K_3$, and all their faces. Then $K_{2,3}$ is pure. To prove Claim~(3) of Theorem~\ref{theorem_smaller_groups}, we need several auxiliary propositions.

\begin{propos}\label{propos_Sym23}
$\Sym(K_{2,3})=G_{351}$.
\end{propos}

\begin{proof}
The proof of this proposition repeats the proof of Proposition~\ref{propos_SymK} (see Section~\ref{section_no_add_sym}) almost literally. Namely, first, $\Sym(K_{2,3})\supseteq G_{351}$, since $\Sym(K_2)=\Sym(K_3)=G_{351}$. Second, as in the proof   of Lemma~\ref{lem_not_transitive}, we can compute the matrix $(N_{pq})$ for~$K_{2,3}$. The result of the computation is shown in Table~\ref{table_N_matrix_23}.
\begin{table}
\caption{The matrix $(N_{pq})$ for~$K_{2,3}$}\label{table_N_matrix_23}
\begin{tabular}{|c|ccccccccc|}
\hline
& 1 & 2 & 3 & 4 & 5 & 6 & 7 & 8 & 9 \\
\hline
1 & 17 & 318 & 241 & 128 & 37 & 4 & 2 &  0 & 0 \\ 
2 & 353 &  3975 &  4336 & 3787 & 1703 & 462 & 133 & 12 & 2 \\ 
3 & 312 & 4671 & 18186 & 23937 & 13733 & 4676 & 1401 & 237 & 28 \\ 
4 & 146 & 3929 & 24376 & 48766 & 23576 & 6948 & 1821 & 259 & 28 \\ 
5 & 41 & 1655 & 14156 & 23770 & 14830 & 4470 & 1095 & 147 & 15 \\ 
6 & 4 & 397 & 4678 & 6910 & 4622 & 1776 & 521 & 78 & 12 \\ 
7 & 0 & 112 & 1356 & 1887 & 1134  & 492 & 190 & 42 & 6 \\ 
8 & 0 & 14 & 228 & 258 & 156 & 74 & 36 & 9 & 0 \\ 
9 & 0 & 1 & 22 & 38 & 15 & 5 & 7 & 1 & 0 \\
\hline
\end{tabular}
\end{table}
Since this matrix is not symmetric, we obtain that $\Sym(K_{2,3})$ does not act transitively on the set of directed edges and hence 
$$
G_{351}\subseteq \Sym(K_{2,3})\subseteq \Sym(\Gamma)=G_{351}\rtimes \rC_3.
$$ 
So the proposition will follow if we show that the Frobenius automorphism~$F$ does not belong to~$\Sym(K_{2,3})$. To this end, note that $K_{2,3}$ contains $259$~\ $G_{351}$-orbits of $16$-simplices, and the intersection~$K_2\cap F(K_2)$ contains only $6$~\ $G_{351}$-orbits of $16$-simplices, see Table~\ref{table_intersections}. Therefore, $K_{2,3}$ is not contained in $K_2\cap F(K_2)$ and hence $F\notin\Sym(K_{2,3})$. Thus, $\Sym(K_{2,3})=G_{351}$.
\end{proof}

It follows immediately from the construction that $K_S$ contains at least $351$ distinguished subcomplexes, namely, the subcomplexes~$g(J)$ with $g\in G_{351}\setminus S$ and the subcomplexes~$g(\widetilde{J}\,)$ with $g\in S$. To prove Claim~(3) of Theorem~\ref{theorem_smaller_groups} we, in particular, need to check that $K_S$ contains no other distinguished subcomplexes. Certainly, we cannot check this directly by computing the $\nu$-parameters for all~$K_S$, since the number of complexes~$K_S$ is huge. Nevertheless, we can compute the $\nu$-parameters for the pure simplicial complex~$K_2\cup K_3$, which contains all combinatorial manifolds~$K_S$. The result of computation is as follows.

\begin{propos}\label{propos_nu_param_calc2+3}
\textnormal{(1)} For each $16$-simplex $\sigma\in K_2\cap K_3$, we have 
 $\nu_{K_2\cup K_3}(\sigma,v)\le 7$ for all $v\notin \sigma$.

\textnormal{(2)} For each $16$-simplex $\sigma\in K_2\cup K_3$ that does not belong to~$K_2\cap K_3$, there exists a unique vertex $v_{\sigma}\notin\sigma$ satisfying $\nu_{K_2\cup K_3}(\sigma,v_{\sigma})= 17$, and $\nu_{K_2\cup K_3}(\sigma,v)\le 7$ for all other vertices $v\notin \sigma$.
\end{propos}

Since $K_S\subset K_2\cup K_3$, we have $\nu_{K_S}(\sigma,v)\le\nu_{K_2\cup K_3}(\sigma,v)$  whenever $\sigma\in K_S$ and $v\notin\sigma$. Besides, $\nu_{K_S}(\sigma,v)\le 8$ by Proposition~\ref{propos_nu_param}(1). Also we know that if $\sigma\notin K_2\cap K_3$, i.\,e., $\sigma$ lies in one of the subcomplexes $g(J)$ and~$g(\widetilde{J}\,)$, then $\sigma$ lies in a distinguished subcomplex of~$K_S$ and hence there is a vertex $v\notin\sigma$ satisfying $\nu_{K_S}(\sigma,v)=8$. This yields the following proposition.  

\begin{propos}\label{propos_nu_param_calcS}
\textnormal{(1)} For each $16$-simplex $\sigma\in K_2\cap K_3$, we have 
 $\nu_{K_S}(\sigma,v)\le 7$ for all $v\notin \sigma$.

\textnormal{(2)} For each $16$-simplex $\sigma\in K_S$ that does not belong to~$K_2\cap K_3$, there exists a unique vertex $v_{\sigma}\notin\sigma$ satisfying $\nu_{K_S}(\sigma,v_{\sigma})= 8$ and $\nu_{K_S}(\sigma,v)\le 7$ for all other vertices $v\notin \sigma$.
\end{propos}

\begin{cor}\label{cor_distinguished}
For each subset $S\subseteq G_{351}$, the combinatorial manifold~$K_S$ contains exactly $351$ distinguished subcomplexes, namely, the subcomplexes~$g(J)$ with $g\in G_{351}\setminus S$  and the subcomplexes~$g(\widetilde{J}\,)$ with $g\in S$.
\end{cor}

\begin{proof}[Proof of Claim~(3) of Theorem~\ref{theorem_smaller_groups}]
It follows from Claim~(2) of Theorem~\ref{theorem_smaller_groups},  Proposition~\ref{propos_move}, and Corollary~\ref{cor_triple_commute} that each~$K_S$ is a well-defined $27$-vertex  combinatorial $16$-manifold like the octonionic projective plane and is PL homeomorphic to~$K_2$.

Suppose that $f\colon K_{S_1}\to K_{S_2}$ is an isomorphism. We conveniently identify~$f$ with the corresponding permutation of vertices. Since any isomorphism of combinatorial manifolds preserves the $\nu$-parameters, it follows from Proposition~\ref{propos_nu_param_calcS} that $f$ takes $16$-simplexes belonging to~$K_2\cap K_3$ to  $16$-simplexes belonging to~$K_2\cap K_3$ and  $16$-simplexes not belonging to~$K_2\cap K_3$ to  $16$-simplexes not belonging to~$K_2\cap K_3$. Hence, $f$ belongs to the subgroup $\Sym(K_{2,3})\subset \rS_{27}$, which by Proposition~\ref{propos_Sym23} coincides with~$G_{351}$. Now, for each $g\in G_{351}\setminus S_1$, the permutation~$f$ takes the subcomplex $g(J)\subset K_{S_1}$ to the complex $fg (J)$, which therefore must be contained in~$K_{S_2}$. Hence, $f(G_{351}\setminus S_1)\subseteq G_{351}\setminus S_2$. Similarly, considering the subcomplexes $g(\widetilde{J}\,)\subset K_{S_1}$ for $g\in S_1$, we obtain that $fS_1\subseteq S_2$. Thus, $fS_1=S_2$. In particular, taking $S_1=S_2=S$, we see that $\Sym(K_S)$ is exactly the subgroup of~$G_{351}$ consisting of all~$f$ satisfying~$fS=S$. These assertions together with Corollary~\ref{cor_distinguished} constitute Claim~(3) of Theorem~\ref{theorem_smaller_groups}.
\end{proof}

\begin{proof}[Proof of Claim~(5) of Theorem~\ref{theorem_smaller_groups}]
Set $G=G_{351}$. The group~$G$ acts by left shifts on itself and hence on subsets $S\subseteq G$. We denote the stabilizer of~$S$ in~$G$ by~$G_S$. By Claim~(3) of Theorem~\ref{theorem_smaller_groups},  $\Sym(K_S)=G_S$  and $K_{S_1}$ is isomorphic to~$K_{S_2}$ if and only if $gS_1=S_2$ for some $g\in G$. Hence, for a subgroup $H\subseteq G$, the number~$m_H$ of isomorphism classes of triangulations~$K_S$ with $\Sym(K_S)$ conjugate to~$H$ is equal to the number of $G$-orbits of subsets $S$ with stabilizers conjugate to~$H$. Therefore, 
\begin{equation}\label{eq_mH}
m_H=\frac{r_Hn_H}{[G:H]}=\frac{n_H}{[N(H):H]}\,,
\end{equation}
where $n_H$ is the number of subsets $S\subseteq G$ such that $G_S=H$, $N(H)$ is  the normalizer of~$H$ in~$G$, and $r_H=[G:N(H)]$ is the number of conjugates of~$H$.

Now, 
\begin{equation}\label{eq_nH}
n_H=n_{\ge H}-\sum_{H\subsetneq Q\subseteq G} n_{Q},
\end{equation}
where $n_{\ge H}$ is the number of subsets $S\subseteq G$ such that  $G_S\supseteq H$. Obviously,  $G_S\supseteq H$ if and only if $S$ is a union of right cosets of~$H$ in~$G$. Hence, 
\begin{equation}\label{eq_n>=H}
n_{\ge H}=2^{[G:H]}.
\end{equation}

Formulae~\eqref{eq_mH}--\eqref{eq_n>=H} allow us to compute the required numbers~$m_H$ for all conjugacy classes of subgroups of~$G$.
It is easy to check that the group $G=\rC_3^3\rtimes \rC_{13}$ contains exactly $6$ conjugacy classes of subgroups:
\begin{enumerate}
\item The whole group $G$. We have $m_G=n_G=n_{\ge G}=2$.
\item The normal subgroup~$\rC_3^3$. We have
\begin{align*}
n_{\rC_3^3}&=n_{\ge \rC_3^3}-n_G=2^{13}-2,\\
m_{\rC_3^3}&=\frac{n_{\rC_3^3}}{13}=630.
\end{align*}
\item The conjugacy class consisting of $27$ subgroups isomorphic to~$\rC_{13}$. We have $N(\rC_{13})=\rC_{13}$ and
$$
m_{\rC_{13}}=n_{\rC_{13}}=n_{\ge \rC_{13}}-n_G=2^{27}-2.
$$

\item The conjugacy class consisting of $13$ subgroups isomorphic to~$\rC_{3}^2$. We have $N(\rC_{3}^2)=\rC_3^3$ and
\begin{align*}
n_{\rC_3^2}&=n_{\ge \rC_3^2}-n_{\rC_3^3}-n_G=2^{39}-2^{13},\\
m_{\rC_3^2}&=\frac{n_{\rC_3^2}}{3}=\frac{1}{3}(2^{39}-2^{13}).
\end{align*}

\item The conjugacy class consisting of $13$ subgroups isomorphic to~$\rC_{3}$. We have $N(\rC_{3})=\rC_3^3$. Besides, each subgroup~$\rC_3$ is contained in exactly $4$ subgroups isomorphic to~$\rC_3^2$. Therefore,
\begin{align*}
n_{\rC_3}&=n_{\ge \rC_3}-4n_{\rC_3^2}-n_{\rC_3^3}-n_G=2^{117}-2^{41}+3\cdot 2^{13},\\
m_{\rC_3^2}&=\frac{n_{\rC_3}}{9}=\frac19(2^{117}-2^{41}+3\cdot 2^{13}).
\end{align*}

\item The trivial subgroup~$1$. We have 
\begin{align*}
n_{1}&=n_{\ge 1}-13n_{\rC_3}-13n_{\rC_3^2}-27n_{\rC_{13}}-n_{\rC_3^3}-n_G\\
{}&=2^{351}-13\cdot 2^{117}+39\cdot 2^{39}- 27\cdot (2^{27}+2^{13}-2),\\
m_{1}&=\frac{n_{1}}{351}=\frac1{351}\bigl(2^{351}-13\cdot 2^{117}+39\cdot 2^{39}- 27\cdot (2^{27}+2^{13}-2)\bigr).
\end{align*}
\end{enumerate}
This completes the proof of Claim~(5) of Theorem~\ref{theorem_smaller_groups}.
\end{proof}

\begin{remark}\label{remark_no_other}
It follows from Theorem~\ref{theorem_main_detail} that $K_1$, $K_2$, $K_3$, and~$K_4$ are, up to isomorphism,  the only $27$-vertex combinatorial $16$-manifolds like the octonionic projective plane with the symmetry group $G_{351}\subset \rS_{27}$. However, Table~\ref{table_num_triang} contains a stronger assertion. Namely, we assert that $K_1$, $K_2$, $K_3$, and~$K_4$ are the only $27$-vertex combinatorial $16$-manifolds like the octonionic projective plane with the symmetry groups isomorphic to~$G_{351}$. To prove this we need to check that the group $G_{351}$ has, up to isomorphism,  a unique effective action on $27$ points. Indeed, all transitive actions of~$G_{351}$ on~$[27]$ are isomorphic, since all index~$27$ subgroups $\rC_{13}\subset G_{351}$ are conjugate to each other. Besides, if an action of~$G_{351}$ on~$[27]$ is not transitive, then all orbits of this action have cardinalities~$1$ or~$13$ and therefore the normal subgroup $\rC_3^3\subset G_{351}$ lies in the kernel of the action. So the action is not effective.
\end{remark}

\section{Fixed point sets}\label{section_fixed}

The combinatorial manifolds~$\RP^2_6$, $\CP^2_9$, $\HP^2_{15}$, and the four $27$-vertex combinatorial $16$-manifolds~$K_1,\ldots, K_4$ constructed above have rich symmetry groups, namely
\begin{gather*}
\Sym(\RP^2_6)\cong\Sym(\HP^2_{15})\cong\rA_5,\qquad
\Sym(\CP^2_9)\cong \mathrm{He}_3\rtimes\rC_2,\\
\Sym(K_i)\cong G_{351}=\rC_3^3\rtimes \rC_{13},\qquad i=1,2,3,4.
\end{gather*}
So it is interesting to study the fixed point sets~$K^H$, where $K$ is one of the listed combinatorial manifolds and $H$ is a subgroup of~$\Sym(K)$. Before doing this, let us recall a standard fact that the fixed point set of a simplicial action of a finite group on a simplicial complex always has a canonical structure of a simplicial complex.

Suppose that $K$ is a finite simplicial complex on vertex set~$V$. Recall that the (\textit{standard}) \textit{geometric realization} of~$K$ is the subset $|K|\subset\R^V$ consisting of all points
\begin{equation*}
\bx=\sum_{v\in V}x_vv
\end{equation*}
such that
\begin{itemize}
 \item $\sum_{v\in V}x_v=1$,
 \item $x_v\ge 0$ for all $v\in V$,
 \item the set of all~$v$ with $x_v>0$ is a simplex of~$K$.
\end{itemize}
The numbers~$x_v$ are called the \textit{barycentric coordinates} of a point~$\bx$. The point
$$
b(\sigma)=\frac{1}{|\sigma|}\sum_{v\in\sigma}v
$$
is called the \textit{barycentre} of a simplex $\sigma\in K$.
Now, suppose that $H$ is a subgroup of~$\Sym(K)$ and $|K|^H$ is the set of all $H$-fixed points in~$|K|$. The following proposition is standard.

\begin{propos}\label{propos_fixed}
 Suppose that $\sigma_1,\ldots,\sigma_m,\nu_1,\ldots,\nu_n$ are all orbits of $H$ acting on~$V$ such that $\sigma_i$ are simplices of~$K$ and $\nu_i$ are non-simplices of~$K$. Then
 the set $|K|^H$ is the geometric realization of the simplicial complex~$K^H$ such that
 \begin{itemize}
  \item the vertices of~$K^H$ are the barycentres $b(\sigma_1),\ldots,b(\sigma_m)$,
  \item a set $\bigl\{ b(\sigma_{i_1}),\ldots,b(\sigma_{i_k})\bigr\}$ spans a simplex of~$K^H$ if and only if $\sigma_{i_1}\cup\cdots\cup\sigma_{i_k}$ is a simplex of~$K$.
 \end{itemize}
\end{propos}

Now, suppose that $K$ is one of the combinatorial manifolds~$\RP^2_6$, $\CP^2_9$, and~$\HP^2_{15}$. Then one can easily list all conjugacy classes of subgroups $H\subset\Sym(K)$ and, for each of them, find the corresponding fixed point complex~$K^H$. (Certainly, the fixed point complexes correponding to conjugate subgroups are isomorphic.) The result is given in Table~\ref{table_fps}. (In this table `$\mathrm{pt}$' denotes a point and `$3\,\mathrm{pts}$' denotes the disjoint union of three points.)  In the most interesting cases $K=\CP^2_{9}$, $H\cong\rC_2$ and $K=\HP^2_{15}$, $H\cong \rC_2$ or~$\rC_2^2$ or~$\rC_3$, the complexes~$K^H$ were found in~\cite[Section~7.2]{BrKu92}; the other cases are easy. Note that the group~$\Sym(\CP^2)\cong \mathrm{He}_3\rtimes \rC_2$ contains four conjugacy classes of subgroups isomorphic to~$\rC_3$, three conjugacy classes of subgroups isomorphic to~$\rC_3^2$, and two conjugacy classes of subgroups isomorphic to either of the groups~$\rS_3$ and~$\rS_3\times\rC_3$. In each of these cases, the second row of the table contains the number of conjugacy classes of subgroups with the specified fixed point complex. For instance, for one conjugacy class of subgroups isomorphic to~$\rC_3$, the corresponding fixed point complex is isomorphic to~$\partial\Delta^3\sqcup\mathrm{pt}$, and for the other three conjugacy classes of subgroups isomorphic to~$\rC_3$, the corresponding fixed point complex is the disjoint union of $3$ points.

\begin{table}
\caption{Fixed point complexes}\label{table_fps}

\begin{tabular}{|c|c|c|c|c|c|c|c|}
 \hline
 $\boldsymbol{H}\vphantom{C^{(}_2}$ & $\rC_2$ & $\rC_2^2$ & $\rC_3$ & $\rS_3$ & $\rC_5$ & $\rC_5\rtimes\rC_2$  & $\rA_4$ \\
 \hline &&&&&&& \\[-1em]
 $\boldsymbol{\left(\RP^2_{6}\right)^{H}}$ & $\partial\Delta^2\sqcup\mathrm{pt}$ & $3\, \mathrm{pts}$ & $\mathrm{pt}$ & $\mathrm{pt}$ & $\mathrm{pt}$ & $\mathrm{pt}$ & $\varnothing$ \\[2pt]
 \hline
\end{tabular}
\bigskip

\begin{tabular}{|c|c|c|c|c|c|c|c|c|c|c|c|}
 \hline
 $\boldsymbol{H}\vphantom{C^{(}_2}$ & $\rC_2$ & \multicolumn{2}{c|}{$\rC_3$} & $\rC_6$ & \multicolumn{2}{c|}{$\rS_3$} & \multicolumn{2}{c|}{$\rC_3^2$} & \multicolumn{2}{c|}{$\rS_3\times \rC_3$} & $\mathrm{He}_3$\\
 \cline{3-4}\cline{6-11}
 & & 1 class & 3 cl. & & 1 cl. & 1 cl. & 1 cl. & 2 cl. & 1 cl. & 1 cl. & \\
 \hline &&&&&&&&&&& \\[-1em]
 $\boldsymbol{\left(\CP^2_9\right)^{H}}$ & $\RP^2_6$ & $\partial\Delta^3\sqcup\mathrm{pt}$ & $3\, \mathrm{pts}$ & $\mathrm{pt}$ & $\partial\Delta^2\sqcup\mathrm{pt}$ & $3\, \mathrm{pts}$ & $3\, \mathrm{pts}$ & $\varnothing$ & $3\, \mathrm{pts}$ & $\varnothing$ & $\varnothing$
 \\[2pt]
 \hline
\end{tabular}
\bigskip

\begin{tabular}{|c|c|c|c|c|c|c|c|}
 \hline
 $\boldsymbol{H}\vphantom{C^{(}_2}$ & $\rC_2$ & $\rC_2^2$ & $\rC_3$ & $\rS_3$ & $\rC_5$ & $\rC_5\rtimes\rC_2$  & $\rA_4$ \\
 \hline &&&&&&& \\[-1em]
 $\boldsymbol{\left(\HP^2_{15}\right)^{H}}$ & $\CP^2_9$ & $\RP^2_6$ & $\partial\Delta^3\sqcup\mathrm{pt}$ & $\partial\Delta^2\sqcup\mathrm{pt}$ & $3\, \mathrm{pts}$ & $3\, \mathrm{pts}$  & $\mathrm{pt}$ \\[2pt]
 \hline
\end{tabular}
\bigskip

\begin{tabular}{|c|c|c|c|c|}
 \hline
 $\boldsymbol{H}\vphantom{C^{(}_2}$ & $\rC_3$ & $\rC_3^2$ & $\rC_3^3$ & $\rC_{13}$ \\
 \hline &&&& \\[-1em]
 $\boldsymbol{K_i^{H}}$ & $\CP^2_9$ & $3\, \mathrm{pts}$ & $\varnothing$ & $3\, \mathrm{pts}$
 \\[2pt]
 \hline
\end{tabular}
\bigskip
\bigskip

\caption{Bijection between vertices~$(x,y)$ of~$\CP^2_9$ and vertices~$b(\sigma)$ of~$K_i^{\rC_3}$}
 \label{table_bij}
 \begin{tabular}{|c|c|c|c|c|c|c|c|}
  \cline{1-2}\cline{4-5}\cline{7-8}
   $\boldsymbol{(x,y)}$ & $\boldsymbol{\sigma}$ & \ \ \  & $\boldsymbol{(x,y)}$ & $\boldsymbol{\sigma}$  & \ \ \   & $\boldsymbol{(x,y)}$ & $\boldsymbol{\sigma}$  \\
  \cline{1-2}\cline{4-5}\cline{7-8}
  $(0,0)$ & $\{1, 14, 27\}$ &  &
  $(1,0)$ & $\{12, 20, 24\}$ & &
  $(2,0)$ & $\{7, 25, 11\}$
  \\
  \cline{1-2}\cline{4-5}\cline{7-8}
  $(0,1)$ & $\{9, 16, 26\}$ & &
  $(1,1)$ & $\{5, 6, 21\}$  & &
  $(2,1)$ & $\{15, 23, 17\}$
  \\
  \cline{1-2}\cline{4-5}\cline{7-8}
  $(0,2)$ & $\{3, 22, 13\}$ & &
  $(1,2)$ & $\{2, 4, 10\}$ & &
  $(2,2)$ & $\{8, 19, 18\}$
  \\
  \cline{1-2}\cline{4-5}\cline{7-8}
 \end{tabular}
\end{table}

Let us now study the simplicial complexes~$K_i^H$, where $K_i$ is one of the four $27$-vertex $16$-dimensional combinatorial manifolds~$K_1$, $K_2$, $K_3$, and~$K_4$, and $H$ is a subgroup of the group $\Sym(K_i) \cong G_{351}$. As we have already mentioned in Section~\ref{section_proof_smaller_groups}, the group~$G_{351}$ contains four conjugacy classes of non-trivial proper subgroups isomorphic to~$\rC_3$, $\rC_3^2$, $\rC_3^3$, and~$\rC_{13}$, respectively. The subgroup~$\rC_3^3$ acts transitively on the vertices of every~$K_i$. Hence, by Proposition~\ref{propos_fixed} we have $K_i^{\rC_3^3}=\varnothing$.

\begin{propos}
 For every $i=1,2,3,4$,
 \begin{enumerate}
 \item the simplicial complex~$K_i^{\rC_3}$ is isomorphic to the K\"uhnel triangulation~$\CP^2_9$,
 \item either of the simplicial complexes~$K_i^{\rC_3^2}$ and~$K_i^{\rC_{13}}$ is a disjoint union of three points.
 \end{enumerate}
\end{propos}

\begin{proof}
The most interesting situation occurs when $H\cong\rC_3$. We may take for~$H$ the subgroup $\rC_3\subset G_{351}$ generated by the permutation
\begin{equation*}
 B = (1\ 14\ 27)(2\ 4\ 10)(3\ 22\ 13)(5\ 6\ 21)(7\ 25\ 11)(8\ 19\ 18)(9\ 16\ 26)(12\ 20\ 24)(15\ 23\ 17).
\end{equation*}
We will only point out explicitly a bijection between the vertices of~$K_i^{\rC_3}$ and the vertices of~$\CP^2_9$. The fact that this bijection gives an isomorphism of these two simplicial complexes is checked directly with a computer using the explicit lists of maximal simplices of~$K_i$ (see Tables~\mbox{\ref{table_1234}--\ref{table_4}}) and Proposition~\ref{propos_fixed}.

 To point out a required bijection we will conveniently use the following description of~$\CP^2_9$ due to Bagchi and Datta~\cite{BaDa94}. The vertices of~$\CP^2_9$ are the points of the affine plane~$\mathcal{P}$ over the three-element field~$\F_3$.  Fix three mutually parallel lines~$l_0$, $l_1$, and~$l_2$ of~$\mathcal{P}$ together with a cyclic order of them. These three lines will be called \textit{special}.  There are two types of four-dimensional simplices of~$\CP^2_9$. Firstly, for any two distinct intersecting lines~$m_1$ and~$m_2$ of~$\mathcal{P}$ neither of which is special, the $5$-element set $m_1\cup m_2$ is a simplex of~$\CP_2^9$. This gives $27$ four-dimensional simplices. Secondly, for $0\le i\le 2$ and for any point $w$ on~$l_i$, the $5$-element set $l_i\cup l_{i+1}\setminus\{w\}$ is a simplex of~$\CP_2^9$. (Here the sum $i+1$ is taken modulo~$3$.) This gives $9$ four-dimensional simplices. Let us now introduce an affine coordinate system~$x,y$ on~$\mathcal{P}$ so that each line~$l_i$ is given by the equation~$y=i$. Then the vertices of~$\CP^2_9$ are indexed by pairs~$(x,y)$ with $x,y\in\F_3$.

 The orbits of the group $\rC_3$ acting on the vertices of~$K_i$ are exactly the $9$ cycles in the decomposition into disjoint cycles of the permutation~$B$. Since any three-element subset spans a simplex of~$K_i$, we see that $K_i^{\rC_3}$ has $9$ vertices~$b(\sigma)$, where $\sigma$ runs over those $9$ cycles. A bijection providing an isomorphism $K_i^{\rC_3}\cong\CP^2_9$ is given in Table~\ref{table_bij}. (The same bijection is suitable for all the four complexes~$K_1$, $K_2$, $K_3$, and~$K_4$.)

 Let us now prove assertion~(2). Any subgroup $\rC_3^2\subset G_{351}$ acts on the vertex set~$[27]$ of every~$K_i$ with three orbits, each of which consists of $9$ elements. Since every $9$-element subset of~$[27]$ is a simplex of~$K_i$ and every $18$-element subset of~$[27]$ is a non-simplex of~$K_i$, by Proposition~\ref{propos_fixed} we obtain  that $K_i^{\rC_3^2}$ is a disjoint union of three points.

Take the subgroup~$\rC_{13}\subset G_{351}$ generated by the permutation
$$
A = (1\ 2\ 3\ 4\ 5\ 6\ 7\ 8\ 9\ 10\ 11\ 12\ 13)(14\ 15\ 16\ 17\ 18\ 19\ 20\ 21\ 22\ 23\ 24\ 25\ 26).
$$
Then $\rC_{13}$ acts on the vertex set~$[27]$ of every~$K_i$ with three orbits, one of which consists of a single vertex~$27$, and the other two are
\begin{align*}
 \sigma_1&=\{1,2, 3, 4, 5, 6, 7, 8, 9, 10, 11, 12, 13\},\\
 \sigma_2&=\{14,15, 16, 17, 18, 19, 20, 21, 22, 23, 24, 25, 26\}.
\end{align*}
From Tables~\mbox{\ref{table_1234}--\ref{table_4}} one can easily see that both~$\sigma_1$ and~$\sigma_2$ are simplices of every~$K_i$. Then by the complementarity we obtain that $\sigma_1\cup\{27\}\notin K_i$ and $\sigma_2\cup\{27\}\notin K_i$. Therefore, by Proposition~\ref{propos_fixed} we see that $K_i^{C_{13}}$ is the disjoint union of the three points~$b(\sigma_1)$, $b(\sigma_2)$, and~$27$.
\end{proof}

We have already mentioned in the Introduction (see Remark~\ref{remark_Ale}) that the group~$G_{351}$ can be realized as a subgroup of the isometry group~$\Isom(\OP^2)\cong F_4$ of the octonionic projective plane~$\OP^2$ endowed with the Fubini--Study metric. In order to provide some evidence towards Conjecture~\ref{conj_main}, it is natural to study the fixed point sets $(\OP^2)^H$ of subgroups~$H\subset G_{351}\subset\Isom(\OP^2)$ and show that they have the same topological types as the corresponding fixed point complexes~$K_i^H$. This is indeed true. We will not give an analysis of all cases here, but will focus on the most interesting case of $H=\rC_3$. (The author plans to devote a separate paper to the connection between the triangulations~$K_1$, $K_2$, $K_3$, and~$K_4$ and the group~$\Isom(\OP^2)\cong F_4$.)

\begin{propos}\label{propos_fixed_OP}
 Suppose that $$\rC_3\subset G_{351}\subset\Isom(\OP^2).$$ Then the fixed point set~$(\OP^2)^{\rC_3}$ is the complex projective plane~$\CP^2$, which is embedded standardly in~$\OP^2$.
\end{propos}

\begin{remark}
 The phrase `embedded standardly' means the following. The sequence of subgroups  $\rC_3\subset G_{351}\subset\Isom(\OP^2)$ is defined up to conjugation. We assert that this sequence can be chosen so that $(\OP^2)^{\rC_3}$ coincides with the standard $\CP^2\subset\OP^2$, where the inclusion is induced by the inclusion~$\C\subset\mathbb{O}$.
\end{remark}

\begin{remark}
 An analogous result for the action of subgroups of~$\rA_5$ on~$\HP^2$ was obtained in~\cite[Section~7.2]{BrKu92}.
\end{remark}

\begin{proof}[Proof of Proposition~\ref{propos_fixed_OP}]
We conveniently realize~$\OP^2$ as the set of all Hermitian matrices
$$
P=\begin{pmatrix}
   \xi_1 & x_3 & \bar{x}_2 \\
   \bar{x}_3 & \xi_2 & x_1 \\
   x_2 & \bar{x}_1 & \xi_3
  \end{pmatrix},
$$
where $\xi_1,\xi_2,\xi_3\in\R$ and~$x_1,x_2,x_3\in\mathbb{O}$,
that satisfy the conditions
$$
P^2=P\quad\text{and}\quad\mathrm{tr}(P)=1,
$$
cf.~\cite[Section~12.2]{CoSm03}.
Choosing a purely imaginary unit octonion~$i$, we obtain an embedding $\C\subset\mathbb{O}$. Consider the splitting
$$
\mathbb{O}=\C\oplus U,
$$
where $U$ is the orthogonal complement of~$\C$ in~$\mathbb{O}$. Note that $U$ is a two-sided $\C$-module and
$
zu=u\bar{z}
$
for all $z\in\C$ and all~$u\in U$. Let us write every octonion~$x_s$ in the form
$$
x_s=z_s+u_s, \qquad z_s\in\C,\ u_s\in U.
$$
Recall that the subgroup
$$
\rC_3^3\subset G_{351}\subset F_4=\Isom(\OP^2)
$$
is the non-toral commutative subgroup constructed by Alekseevskii~\cite{Ale74}, see Remark~\ref{remark_Ale}. From the construction of this subgroup it follows that any subgroup~$\rC_3$ of it is conjugate in~$\Isom(\OP^2)$ to the subgroup generated by the isometry
\begin{equation}\label{eq_formula_C3}
\gamma\colon \begin{pmatrix}
              \xi_1 & z_3+u_3 & \bar{z}_2-u_2 \\
              \bar{z}_3-u_3 & \xi_2 & z_1+u_1 \\
              z_2+u_2 & \bar{z}_1-u_1 & \xi_3
             \end{pmatrix}
             \mapsto
             \begin{pmatrix}
              \xi_1 & z_3+\omega u_3 & \bar{z}_2-\omega u_2 \\
              \bar{z}_3-\omega u_3 & \xi_2 & z_1+\omega u_1 \\
              z_2+ \omega u_2 & \bar{z}_1-\omega u_1 & \xi_3
             \end{pmatrix},
\end{equation}
where $\omega=\exp(2\pi i/3)$ is a primitive cubic root of unity in~$\C\subset\mathbb{O}$.
We see that $\gamma(P)=P$ if and only if $u_1=u_2=u_3=0$. So the fixed point set of~$\gamma$ is the standard $\CP^2\subset\OP^2$.
\end{proof}

\begin{remark}
Let us explain in more detail where the formula~\eqref{eq_formula_C3} comes from. Alekseevskii's construction of a non-toral subgroup~$\rC_3^3$ of~$F_4$ is as follows.
 There is a homomorphism
$$
\varphi\colon \SU(3)\times \SU(3)\to F_4
$$
whose kernel is the diagonal subgroup $\Delta\cong\rC_3$ of the centre
$$
Z\bigl(\SU(3)\times\SU(3)\bigr)\cong \rC_3\times\rC_3.
$$
An explicit description of this homomorphism can be found in~\cite[Theorem~2.9]{Yok85}. Each factor~$\SU(3)$ contains the subgroup~$\rC_3^2$ generated by the matrices
$$
\begin{pmatrix}
 \omega & 0 & 0\\
 0 & \omega & 0\\
 0 & 0 & \omega
\end{pmatrix}
\quad\text{and}\quad
\begin{pmatrix}
 0 & 1 & 0\\
 0 & 0 & 1\\
 1 & 0 & 0
\end{pmatrix}.
$$
Then the required subgroup $\rC_3^3\subset F_4$ is $\varphi(\rC_3^2\times\rC_3^2)$. We are interested in subgroups of~$\rC_3^3$ that are isomorphic to~$\rC_3$. All such subgroups are conjugate in~$G_{351}$ and hence in~$F_4$, so we can choose any of them. The simplest choice is  the subgroup~$\varphi\bigl(Z\bigl(\SU(3)\times\SU(3)\bigr)\bigr)$. The fact that this subgroup is generated by the isometry~$\gamma$ given by~\eqref{eq_formula_C3} follows immediately from an explicit formula for~$\varphi$ provided in~\cite[Section~2.3]{Yok85}.
\end{remark}

\section{Vertex links and bistellar moves}\label{section_links}

Suppose that $K$ is a combinatorial $n$-manifold on the vertex set~$V$. Assume that $\sigma\in K$ is an $r$-simplex such that
$$
\link(\sigma,K)=\partial\tau
$$
for some $(n-r)$-simplex~$\tau\notin K$. Then, replacing the subcomplex~$\sigma*\partial\tau$ of~$K$ with the complex~$\partial\sigma*\tau$, we arrive at a new combinatorial manifold~$K_1$, which is PL homeomorphic to~$K$. The described operation is called a \textit{bistellar move} (of a \textit{bistellar flip} or a \textit{Pachner move}) associated with the simplex~$\sigma$. We refer to bistellar moves associated with codimension~$k$ simplices as to $k$-\textit{moves}.

If either~$\sigma$ or~$\tau$ is zero-dimensional, that is, a vertex~$u$, we conveniently agree that $\partial u=\varnothing$ and $\rho*\varnothing=\rho$ for any simplex~$\rho$. Then the $0$-move associated with an $n$-dimensional simplex~$\sigma$ is the \textit{stellar subdivision} of~$\sigma$, that is, the operation consisting of inserting a new vertex~$u$ inside~$\sigma$ and  replacing~$\sigma$ with the cone $u*\partial\sigma$. Vice versa, the $n$-move associated with a vertex~$u$ such that $\link(u,K)=\partial\tau$ and $\tau\notin K$ is the \textit{inverse stellar subdivision} of~$\tau$, that is, the operation consisting of deleting the vertex~$u$ and  replacing~$u*\partial\tau$ with the simplex~$\tau$.

For each $n$-simplex~$\sigma$ of~$K$, there is a $0$-move associated with it. On the other hand, the bistellar move associated with a simplex of positive codimension is not always defined.

\begin{theorem}[Pachner \cite{Pac87}]
Two compact combinatorial manifolds~$K_1$ and~$K_2$ are PL homeomorphic to each other if and only if $K_1$ can be taken to~$K_2$ by a sequence of bistellar moves and simplicial isomorphisms.
\end{theorem}

For more detail on bistellar moves and Pachner's theorem, see~\cite{Lic99}.

\begin{cor}
 Any $n$-dimensional combinatorial sphere can be taken by a sequence of bistellar moves and simplicial isomorphisms to~$\partial\Delta^{n+1}$, where $\Delta^{n+1}$ is the standard $(n+1)$-dimensional simplex.
\end{cor}

Note that, for $n\ge 3$, it is not true at all that any combinatorial sphere can be taken to the boundary of a simplex \textit{monotonically} with respect to the number of vertices, that is, without the use of $0$-moves. Moreover,  Nabutovsky~\cite{Nab96} showed that Novikov's theorem on the algorithmic unrecognazibility of the $n$-dimensional sphere for $n\ge 5$  (see~\cite[Section~10]{VKF74}) implies the following assertion.

\begin{propos}\label{propos_nab}
 For each~$n\ge 5$ and each Turing computable function~$\theta\colon\Z_{>0}\to \Z_{>0}$, there exists a positive integer~$m$ and an $n$-dimensional combinatorial sphere~$L$ with~$m$ vertices such that every sequence of bistellar moves taking~$L$ to~$\partial\Delta^{n+1}$ contains a combinatorial sphere with at least~$\theta(m)$ vertices.
\end{propos}

On the other hand, Pachner~\cite{Pac78} proved that any \textit{polytopal sphere} can be taken by bistellar moves to the boundary of a simplex monotonically with respect to the number of vertices. (A combinatorial sphere is \textit{polytopal} if it is isomorphic to the boundary of a simplicial convex polytope.) It is interesting to consider the following class of combinatorial spheres, which in a certain sense are the furthest from being polytopal.

\begin{defin}[\cite{DFM04}]
 A combinatorial sphere~$L$ is said to be \textit{unflippable} if no bistellar move other than a $0$-move (i.\,e. a vertex insertion) is possible for~$L$.
\end{defin}

Note that the existence of unflippable combinatorial spheres does not follow from Proposition~\ref{propos_nab}. The first example was a $16$-vertex three-dimensional unflippable combinatorial sphere constructed by Dougherty, Faber, and Murphy~\cite{DFM04}. Examples of unflippable combinatorial spheres of all dimensions greater than~$3$ were obtained in~\cite{BaDa13}. As far as the author knows, there are no other examples in the literature.

Let us show that the vertex links of our $27$-vertex combinatorial manifolds like the octonionic projective plane provide new examples of unflippable combinatorial spheres. Let $K_1$, $K_2$, $K_3$, and~$K_4$ be the four $27$-vertex combinatorial $16$-manifolds from Theorem~\ref{theorem_main_detail}  (cf. Tables~\ref{table_1234}--\ref{table_4}), and let $L_i$ be the link of a vertex of~$K_i$ for $i=1,2,3,4$. (Since the combinatorial manifolds~$K_i$ are vertex-transitive, the links of all vertices of every triangulation~$K_i$ are isomorphic to each other.) Then $L_1$, $L_2$, $L_3$, and~$L_4$ are $26$-vertex $15$-dimensional combinatorial spheres.

\begin{theorem}
The combinatorial spheres~$L_1$ and~$L_4$ are unflippable.
\end{theorem}

\begin{proof}
 Suppose that $i$ is either~$1$ or~$4$ and $L_i=\link(v,K_i)$. Assume that a bistellar $k$-move with $k>0$ is possible for~$L_i$. Then there exist a simplex $\sigma\in L_i$ and non-simplex $\tau\notin L_i$ such that $\dim\sigma=15-k$, $\dim\tau=k$, and
 $$\link(\sigma,L_i)=\partial\tau.$$
 Since $K_i$ is $9$-neighborly, we see that $L_i$ is $8$-neighborly. Hence, $L_i$ contains no $k$-dimensional non-simplices with $0<k\le 7$. Therefore, $k\ge 8$.

 Assume that $k>8$; then $\dim\sigma<7$. From the $8$-neighborliness of~$L_i$ it follows that $\sigma\cup\{w\}\in L_i$ for any vertex~$w$ of~$L_i$. Hence, all vertices of~$L_i$  belong to the star of~$\sigma$. Therefore, the link of~$\sigma$ has
 $$
 26-(\dim\sigma+1)=k+10
 $$
 vertices and so is not isomorphic to the boundary of a $k$-simplex, which contradicts our assumption.

 Thus, $k=8$. Hence, $\dim\sigma=7$ and $\dim\tau=8$. Then $\sigma\cup\{v\}$ is an $8$-dimensional simplex of~$K_i$ and
 $$
 \link(\sigma\cup\{v\},K_i)=\partial\tau.
 $$
 A contradiction with Corollary~\ref{cor_no8-8} completes the proof of the theorem.
\end{proof}

\begin{cor}
 The combinatorial spheres~$L_1$ and~$L_4$ are non-polytopal.
\end{cor}

\begin{remark}
 The combinatorial spheres~$L_2$ and~$L_3$, as well as all links of vertices of all combinatorial manifolds~$K_S$ from Theorem~\ref{theorem_smaller_groups} admit bistellar $8$-moves and hence are not unflippable. Indeed, each of the combinatorial manifolds~$K_S$ contains a distinguished triple~$(\Delta_1,\Delta_2,\Delta_3)$ and, if $v\in \Delta_i$, then the bistellar move associated with~$\Delta_i\setminus\{v\}$ is possible for~$\link v$. Similarly, the link of a vertex of the K\"uhnel triangulation~$\CP^2_9$ admits a bistellar $2$-move and the links of vertices of the three Brehm--K\"uhnel triangulations of~$\HP^2$ admit bistellar $4$-moves.  Nevertheless, it is known that the links of vertices of~$\CP^2_9$ and~$\HP^2_{15}$ are non-polytopal, see~\cite{KuBa83} and~\cite[Section~5]{BrKu92}, respectively.
\end{remark}

\begin{conj}
 The combinatorial spheres~$L_2$ and~$L_3$ are non-polytopal.
\end{conj}

Alongside with polytopal spheres, there are other important classes of combinatotial spheres.   We will focus on the following two of them.

An $n$-dimensional combinatorial sphere is called \textit{shellable} if its $n$-dimensional simplices can be added one by one in some order $\sigma_1,\ldots,\sigma_q$ so that at each intermediate step~$k$, where $1\le k<q$, the union $\sigma_1\cup\cdots\cup\sigma_k$ is a combinatorial ball. Any polytopal sphere is shellable (see~\cite{BrMa71}), but not vice versa.

 An $n$-dimensional combinatorial sphere is called \textit{star-shaped} if there exists an embedding $\mathrm{cone}(K)\hookrightarrow \R^{n+1}$ whose restriction to every simplex of~$\mathrm{cone}(K)$ is affine linear. Any polytopal sphere is star-shaped, but not vice versa.

The link of any vertex of~$\CP^2_9$ is the \textit{Br\"uckner sphere}~$\mathcal{M}$, one of the two $8$-vertex three-dimensional combinatorial spheres that are non-polytopal, see~\cite{KuBa83}. Though~$\mathcal{M}$ is non-polytopal, in some ways it is not so far from being polytopal, namely,
\begin{itemize}
 \item $\mathcal{M}$ can be taken to~$\partial\Delta^4$ by bistellar moves monotonically with respect to the number of vertices, see~\cite[Section~3.5]{Gai04},
 \item $\mathcal{M}$ is shellable, see~\cite[Section~6]{DaKl78},
 \item $\mathcal{M}$ is star-shaped.
\end{itemize}
The latter assertion follows immediately from the representation of~$\mathcal{M}$ as a $3$-diagram, see~\cite{GrSr67}. It is interesting whether the links of vertices of $(3d/2+3)$-vertex combinatorial $d$-manifolds like projective planes have the same properties for $d=8$ and~$d=16$. Let $\mathcal{N}$ be the link of a vertex of~$\HP^2_{15}$.

\begin{question}
\textnormal{(1)} Can $\mathcal{N}$ be taken to~$\partial\Delta^8$ by bistellar moves monotonically with respect to the number of vertices? Can $L_2$ and~$L_3$ be taken to~$\partial\Delta^{16}$ by bistellar moves monotonically with respect to the number of vertices?

\textnormal{(2)} Are $\mathcal{N}$, $L_1$, $L_2$, $L_3$, and~$L_4$ shellable?

\textnormal{(3)} Are $\mathcal{N}$, $L_1$, $L_2$, $L_3$, and~$L_4$ star-shaped?
\end{question}

Certainly, $L_1$ and~$L_4$ cannot be taken to~$\partial\Delta^8$ by bistellar moves  monotonically, since they are unflippable.

\clearpage

\appendix
\section{Representatives of the $G_{351}$-orbits of $16$-simplices of the combinatorial manifolds~$K_1$, $K_2$, $K_3$, and~$K_4$}

\begin{table}[!ht]
\caption{Representatives of the $112$ orbits of $16$-simplices that enter all the four triangulations $K_1$, $K_2$, $K_3$, and~$K_4$}\label{table_1234}
{\footnotesize
\begin{align*}
111111111110111011100000000&& 111110001111110111111000000&& 111110001111110101111100000&\\
111111110110111111100000000&& 110110101111110111111000000&& 100111100111111101111100000&\\
011111111110111111100000000&& 100111101111110111111000000&& 111111111000110011111100000&\\
011111011111111111100000000&& 110111000111111111111000000&& 111011111010110011111100000&\\
111110111111010111110000000&& 110011100111111111111000000&& 111011101011110011111100000&\\
111111111010110111110000000&& 111111011111111011000100000&& 111111101100011011111100000&\\
111110111110110111110000000&& 110111111111111011000100000&& 111111100101011011111100000&\\
111111101011110111110000000&& 111011011111111011100100000&& 111101101101011011111100000&\\
111110101111110111110000000&& 110011111111111011100100000&& 111100101101111011111100000&\\
111110111111110110101000000&& 010111011111111111100100000&& 111111100001110111111100000&\\
101111111111110110101000000&& 010011111111111111100100000&& 111011100011110111111100000&\\
111111111001110111101000000&& 011111111110111100110100000&& 110010100111111111111100000&\\
111110111101110111101000000&& 111111111110110010110100000&& 111111111101011011001010000&\\
101111111101110111101000000&& 111111111010110101110100000&& 101111111110011011101010000&\\
111011111011110111101000000&& 111111101011110101110100000&& 110111110110111100111010000&\\
110111111011110111101000000&& 111011111011110101101100000&& 110111110010111101111010000&\\
110111110111110111101000000&& 111011111011110011101100000&& 110111100011111101111010000&\\
110110111111110111101000000&& 101011111111110011101100000&& 110110111111111011000110000&\\
100111111111110111101000000&& 111110101111110110011100000&& 100111111111111011000110000&\\
111110101111110110111000000&& 101111101111110110011100000&& 111111001111111001100110000&\\
101111101111110110111000000&& 111111101101011011011100000&& 111111000111111011100110000&\\
110111100111111101111000000&& 111111101001110111011100000&& 111011010111111011100110000&\\
111110111101010111111000000&& 111110101101110111011100000&& 110011110111111011100110000&\\
011111111101010111111000000&& 101111101101110111011100000&& 111011010011111111100110000&\\
110110111111010111111000000&& 111111001011110111011100000&& 110011110011111111100110000&\\
111111111000110111111000000&& 111110001111110111011100000&& 111111001111011001110110000&\\
111110111100110111111000000&& 110110101111110111011100000&& 111110101101111001110110000&\\
011111111100110111111000000&& 100111101111110111011100000&& 111111001011111001110110000&\\
111011111010110111111000000&& 101111101111110100111100000&& 011110101101111101110110000&\\
110111111010110111111000000&& 001111111110111010111100000&& 010111110110111101101110000&\\
110110111110110111111000000&& 101111110011110110111100000&& 010111110110111011101110000&\\
111111101001110111111000000&& 111111111000110101111100000&& 011111111100011011011110000&\\
111110101101110111111000000&& 111011111010110101111100000&& 111111001011010111011110000&\\
101111101101110111111000000&& 111111101001110101111100000&& 011111111100011010111110000&\\
111111001011110111111000000&& 111110101101110101111100000&& 001111110110111010111110000&\\
111011101011110111111000000&& 101111101101110101111100000&& 110111110010111001111110000&\\
110111101011110111111000000&& 111111001011110101111100000&& &\\
110111100111110111111000000&& 111011101011110101111100000&& &
\end{align*}
}
\end{table}

\begin{table}[p]
\caption{Representatives of the $89$ orbits of $16$-simplices that enter $K_1$, $K_2$, and $K_3$ but do not enter~$K_4$}\label{table_123}
{\footnotesize
\begin{align*}
111111111111011111000000000&& 111010111011110111111000000&& 111111101011110110011100000&\\
111111101111111111000000000&& 110010111111110111111000000&& 111111111000110111011100000&\\
111011111111111111000000000&& 111100001111111111111000000&& 111111111010110100111100000&\\
111111111111011011100000000&& 110101111111111011100100000&& 111111101011110100111100000&\\
111111101111111011100000000&& 111111101101111011010100000&& 111110101111110100111100000&\\
111011111111111011100000000&& 111111001111111011010100000&& 001111111110111100111100000&\\
111111110111011111100000000&& 110111101111111011010100000&& 111111111010110010111100000&\\
011111111111011111100000000&& 111111101101111001110100000&& 111111111010010110111100000&\\
111111100111111111100000000&& 111111001111111001110100000&& 111111101011010110111100000&\\
111011110111111111100000000&& 111111101100111011110100000&& 111111110010110110111100000&\\
101111101111111111100000000&& 011111111100111011110100000&& 111111100011110110111100000&\\
011111101111111111100000000&& 111111100101111011110100000&& 110110101111110101111100000&\\
011011111111111111100000000&& 111101101101111011110100000&& 011111111100011011111100000&\\
111010111111111101110000000&& 111011101101111011110100000&& 100011100111111111111100000&\\
111111101101111011110000000&& 111111000111111011110100000&& 111111111100011011101010000&\\
111111001111111011110000000&& 111011001111111011110100000&& 110111100111111100111010000&\\
110111101111111011110000000&& 110011101111111011110100000&& 110111010110111110111010000&\\
111010111111110111110000000&& 111110000111111111110100000&& 111111101101111001100110000&\\
110111100111111111110000000&& 111111111011110110001100000&& 111111010111110011100110000&\\
111111111011110110101000000&& 111111111001110111001100000&& 111111100101111011100110000&\\
011111111101110111101000000&& 010111111111110111001100000&& 100111111111011110001110000&\\
010111111111110111101000000&& 111111111011010110101100000&& 110111100111111011001110000&\\
111010111111110111011000000&& 111101111011110110101100000&& 111111111100010011101110000&\\
111111111010110110111000000&& 111111111001110101101100000&& 111111101100011011101110000&\\
111110111110110110111000000&& 010111111111110101101100000&& 011111111100011011101110000&\\
111111101011110110111000000&& 111101011111011011101100000&& 110111100110111011101110000&\\
111010111111110110111000000&& 010111011111110111101100000&& 110111100101111011101110000&\\
110111111011010111111000000&& 010101111111110111101100000&& 010011100111111111101110000&\\
110111110111010111111000000&& 010011111111110111101100000&& 111110101111010100111110000&\\
111010111101110111111000000&& 111111111010110110011100000&&&
\end{align*}
}
\vspace{-2mm}

\caption{Representatives of the $23$ orbits of $16$-simplices that enter $K_1$ and $K_2$ but neither~$K_3$ nor~$K_4$}\label{table_12}
{\footnotesize
\begin{align*}
111111111111110111000000000&& 111110111111010110111000000&& 111110101111111100110010000&\\
111111111111110011100000000&& 110111110111110110111000000&& 111110101101111101110010000&\\
111111111011110111100000000&& 110111110110110111111000000&& 111100101101111101111010000&\\
111111110111110111100000000&& 111101011111111011100100000&& 110110101111111111000110000&\\
111110111111110111100000000&& 111111101101011001111100000&& 110111110111110011001110000&\\
101111111111110111100000000&& 110110100111111101111100000&& 100011111111011111001110000&\\
011111111111110111100000000&& 111111011010110011111100000&& 110111110110110011101110000&\\
111110101111111101110000000&& 111110101111111101010010000&&&
\end{align*}
}
\end{table}

\begin{table}[!ht]
\caption{Representatives of the $62$ orbits of $16$-simplices that enter~$K_1$ but do not enter $K_2$ or $K_3$ or~$K_4$}\label{table_1}
{\footnotesize
\begin{align*}
100111101111111111110000000&& 111111111111010010101100000&& 111111001111110011011100000&\\
111111111111110010101000000&& 111111111110110010101100000&& 111101001111111011011100000&\\
011111111111110110101000000&& 111111110111110010101100000&& 111111101111010010111100000&\\
111111111101110011101000000&& 111101111111110010101100000&& 111111101110110010111100000&\\
111111011111110011101000000&& 111011111111110010101100000&& 111111100111110010111100000&\\
110111111111110011101000000&& 011101111111110110101100000&& 111101101111110010111100000&\\
111111101111110010111000000&& 111111111101110001101100000&& 111011101111110010111100000&\\
011111111110110110111000000&& 111111011111110001101100000&& 111111101101110001111100000&\\
111111101101110011111000000&& 110111111111110001101100000&& 111111001111110001111100000&\\
111111001111110011111000000&& 111111111100110011101100000&& 111101001111111001111100000&\\
110111101111110011111000000&& 111111011110110011101100000&& 111111101100110011111100000&\\
111101001111111011111000000&& 111111110101110011101100000&& 111111001110110011111100000&\\
111111111111110110000100000&& 111011111101110011101100000&& 111111100101110011111100000&\\
110111111111111110000100000&& 111111010111110011101100000&& 111011101101110011111100000&\\
100111101111111111010100000&& 111101011111110011101100000&& 111111000111110011111100000&\\
111111111111110010001100000&& 111011011111110011101100000&& 111011001111110011111100000&\\
011111111111110110001100000&& 110101111111110011101100000&& 111101001110111011111100000&\\
111111111101110011001100000&& 110011111111110011101100000&& 111101001101111011111100000&\\
111111011111110011001100000&& 111101001111111011101100000&& 111101000111111011111100000&\\
110111111111110011001100000&& 111111101111110010011100000&& 111001001111111011111100000&\\
011111111111110100101100000&& 111111101101110011011100000&&&
\end{align*}
}
\vspace{.5cm}
\caption{Representatives of the $37$ orbits of $16$-simplices that enter $K_2$, $K_3$, and $K_4$ but do not enter~$K_1$}\label{table_234}
{\footnotesize
\begin{align*}
111111111110111010110000000&& 011111111110110110110100000&& 101111101111110010111100000&\\
111111111010111011110000000&& 111110111111110010101100000&& 111101001111111010111100000&\\
101011111111110111101000000&& 101111111111110010101100000&& 111110100101111101111100000&\\
110111100111111111011000000&& 001111111111110110101100000&& 111111001111010011111100000&\\
111111101101011011111000000&& 111111111101010011101100000&& 111111101001110011111100000&\\
111001111011110111111000000&& 111111011111010011101100000&& 111110101101110011111100000&\\
101011101111110111111000000&& 111111111001110011101100000&& 101111101101110011111100000&\\
111111111111110010100100000&& 101111111101110011101100000&& 111111001011110011111100000&\\
011111111111110110100100000&& 111111011011110011101100000&& 111110001111110011111100000&\\
111111111101110011100100000&& 111110011111110011101100000&& 101111001111110011111100000&\\
111111011111110011100100000&& 110110111111110011101100000&& 111101001111011011111100000&\\
110111111111110011100100000&& 111011110011110111101100000&& &\\
111111101111110010110100000&& 111110101111110010111100000&& &
\end{align*}
}
\end{table}

\begin{table}[!ht]
\caption{Representatives of the $21$ orbits of $16$-simplices that enter $K_2$ and $K_3$ but neither~$K_1$ nor~$K_4$}\label{table_23}
{\footnotesize
\begin{align*}
111111111111011101100000000&& 111010011111110111111000000&& 111111111011110010101100000&\\
111111101111111101100000000&& 101010111111110111111000000&& 011111111011110110101100000&\\
111011111111111101100000000&& 111111101101110011110100000&& 010111111111110110101100000&\\
111101111111011111100000000&& 111111001111110011110100000&& 110111111111010011101100000&\\
101111111111011111100000000&& 110111101111110011110100000&& 010111111111110011101100000&\\
111101101111111111100000000&& 111101001111111011110100000&& 111111101011110010111100000&\\
111001111111111111100000000&& 110101101111111011110100000&& 011101001111111011111100000&
\end{align*}
}
\vspace{5mm}

\caption{Representatives of the $4$ orbits of $16$-simplices that enter~$K_2$ but do not enter $K_1$ or $K_3$ or~$K_4$}\label{table_2}
{\footnotesize
\begin{align*}
111111111111110101100000000&& 110111110111110011101100000&\\
111101111111110111100000000&& 101101001111111011111100000&
\end{align*}
}
\vspace{5mm}

\caption{Representatives of the $22$ orbits of $16$-simplices that enter $K_3$ and $K_4$ but neither~$K_1$ nor~$K_2$}\label{table_34}
{\footnotesize
\begin{align*}
111111111111111110000000000&& 011111011111111011100100000&& 011110101111111101110010000&\\
111111111111110110100000000&& 111011111011110111001100000&& 110111110111110011101010000&\\
111111111110110111100000000&& 111100101101111101111100000&& 110111110111110011100110000&\\
111111111101110111100000000&& 111011101101011011111100000&& 110111110111110010110110000&\\
111111011111110111100000000&& 111110101111111001110010000&& 111111101101011001110110000&\\
110111111111110111100000000&& 111110101111110101110010000&& 111100101101111101110110000&\\
111110111111011101110000000&& 111110100111111101110010000&& &\\
111011111101111011100100000&& 111100101111111101110010000&& &
\end{align*}
}
\vspace{5mm}

\caption{Representatives of the $5$ orbits of $16$-simplices that enter~$K_3$ but do not enter $K_1$ or $K_2$ or~$K_4$}\label{table_3}
{\footnotesize
\begin{align*}
111111111111010111100000000&& 111011111111110111100000000&& 110111100111110011101110000&\\
111111101111110111100000000&& 011111011100111011111100000&&&
\end{align*}
}
\end{table}

\begin{table}[p]
\caption{Representatives of the $115$ orbits of $16$-simplices that enter~$K_4$ but do not enter $K_1$ or $K_2$ or~$K_3$}\label{table_4}
{\footnotesize
\begin{align*}
111111111110111111000000000&& 111101111001110111111000000&& 111110111011110110101100000&\\
111111011111111111000000000&& 111010110111110111111000000&& 111011111011110110101100000&\\
110111111111111111000000000&& 101111001111110111111000000&& 111111111100011011101100000&\\
111111111111011110100000000&& 111010101111110111111000000&& 011111111100111011101100000&\\
111111101111111110100000000&& 111000111111110111111000000&& 010110111111110111101100000&\\
111011111111111110100000000&& 111110000111111111111000000&& 000111111111110111101100000&\\
111111111110011111100000000&& 010111111111111111000100000&& 011111111100111011011100000&\\
111111111101011111100000000&& 111111111011110110100100000&& 101111001111110111011100000&\\
111111011111011111100000000&& 111111111100111011100100000&& 011111111100111010111100000&\\
111110111111011111100000000&& 111111101101111011100100000&& 111111111000110110111100000&\\
110111111111011111100000000&& 111011111011111011100100000&& 111111011010110110111100000&\\
111111101110111111100000000&& 111111001111111011100100000&& 111011111010110110111100000&\\
111011111110111111100000000&& 110111101111111011100100000&& 111111101001110110111100000&\\
111111101101111111100000000&& 011111111101110111100100000&& 111111001011110110111100000&\\
111011111101111111100000000&& 010111111111110111100100000&& 111110101011110110111100000&\\
111111101011111111100000000&& 111111111100111010110100000&& 111011101011110110111100000&\\
111011111011111111100000000&& 111111101101111010110100000&& 011111111100111001111100000&\\
111111001111111111100000000&& 111111001111111010110100000&& 101111001111110101111100000&\\
111110101111111111100000000&& 110111101111111010110100000&& 011111111000111011111100000&\\
110111101111111111100000000&& 111111111010110110110100000&& 011111101100111011111100000&\\
111011011111111111100000000&& 111111101011110110110100000&& 111011110010110111111100000&\\
111010111111111111100000000&& 011101111110111110110100000&& 111111101011111001110010000&\\
110011111111111111100000000&& 111111101101011011110100000&& 101111001011111111110010000&\\
111111101111111010110000000&& 111111001111011011110100000&& 101111111111011110001010000&\\
111111110110111110110000000&& 110111101111011011110100000&& 101110111111011111001010000&\\
011111111110111110110000000&& 111111111000111011110100000&& 110111100111111011101010000&\\
111111101011111011110000000&& 111011111010111011110100000&& 011111111101011011011010000&\\
111110101111111011110000000&& 111111101001111011110100000&& 111110101111110100111010000&\\
111111111101011011101000000&& 111110101101111011110100000&& 110111100111111011100110000&\\
111111011011110111101000000&& 101111101101111011110100000&& 101111110111110010110110000&\\
111110011111110111101000000&& 011111101101111011110100000&& 110111100111111010110110000&\\
101111011111110111101000000&& 111111001011111011110100000&& 010111110111111011001110000&\\
111010111111110111101000000&& 111011101011111011110100000&& 111111111100011010101110000&\\
111010111111110101111000000&& 111110001111111011110100000&& 111111110100011011101110000&\\
011111111100111011111000000&& 101111001111111011110100000&& 100111100111111011101110000&\\
110111100111111011111000000&& 110110101111111011110100000&& 010111100111111011101110000&\\
111111111001010111111000000&& 111111000011111111110100000&& 111110101101110100111110000&\\
111010111111010111111000000&& 111111111001110110101100000&& &\\
111010111110110111111000000&& 111111011011110110101100000&& &
\end{align*}
}
\end{table}

\end{document}